\documentclass[letterpaper,11pt]{article}
\usepackage[letterpaper,left=1in,right=1in,top=1.5in,bottom=1.5in]{geometry}
\usepackage{bbm}
\usepackage{relsize}

\usepackage{amsmath,amssymb,amsthm, float}

\newcommand{\myauthor}{}

\newcommand{\mytitle}{Knots, primes and class field theory}

\title{\mytitle}

\author{Alain Connes and Caterina Consani\thanks{Partially supported by the Simons Foundation collaboration grant n.691493}}
\date{}

\usepackage[T1]{fontenc}
\usepackage[english]{babel}

\usepackage{amssymb,url,xspace,subfig}

\usepackage{tikz-cd}

\tikzset{%
  symbol/.style={
    draw=none,
    every to/.append style={
      edge node={node [sloped, allow upside down, auto=false]{$#1$}}
    },
  },
}

\usepackage{soul}  

\usepackage[export]{adjustbox}

\usepackage[hidelinks]{hyperref}
\usepackage[
]
{cleveref}

\usepackage{amscd}
\usepackage{amsbsy}
\usepackage{amssymb}
\usepackage{verbatim}
\usepackage{eucal}
\usepackage{microtype}
\usepackage{mathrsfs}
\usepackage{amsthm}
\usepackage{stmaryrd}
\usepackage[all,cmtip]{xy}
\usepackage{tikz}
\usetikzlibrary{matrix,arrows}
\usepackage[abbrev,lite,alphabetic]{amsrefs}
\usepackage{dsfont}

\def\ie{{\it i.e.\/}\ }
\def\qqq{\,,\,~\forall}

\def\Spec{{\rm Spec\,}}
\def\Fr{{\rm Frob}}
\newcommand{\GL}{{\rm GL}}

\def\A{{\mathbb A}}
\def\C{{\mathbb C}}
\def\F{{\mathbb F}}
\def\N{{\mathbb N}}
\def\Q{{\mathbb Q}}
\def\R{\mathbb{R}}
\def\Z{\mathbb{Z}}

\def\cS{\mathcal{S}}
\def\sheaf{{\mathscr{O} }}

\theoremstyle{plain}
\newtheorem{thm}{Theorem}[section]
\newtheorem{defn}[thm]{Definition}
\newtheorem{prop}[thm]{Proposition}
\newtheorem{fact}[thm]{Fact}
\newtheorem{cor}[thm]{Corollary}
\newtheorem{lem}[thm]{Lemma}
\newtheorem{rem}[thm]{Remark}
\newtheorem*{prop*}{Proposition}

\newtheorem*{theorem1}{Theorem 1}
\newtheorem*{theorem2}{Theorem 2}
\newtheorem*{theorem3}{Theorem 3}

\begin{document}
\hypersetup{linkcolor=black}

\maketitle

\begin{abstract}
    \noindent In this paper, we present a geometric generalization of class field theory, demonstrating how adelic constructions, central to the spectral realization of zeros of L-functions and the geometric framework for explicit formulas in number theory, naturally extend the classical theory. This generalization transitions from the idele class group, which acts as the adelic analog of Galois groups, to a geometric framework associated with schemes and the ring of integers of global fields. This perspective provides a conceptual explanation for the role of the adele class space in the spectral realization of L-function zeros and identifies the idele class group as a generic point in this context. The sector $X_{\mathbb{Q}}$ of the adele class space corresponding to the Riemann zeta function gives the class field counterpart of the scaling topos. The main result is the construction of a functor mapping finite abelian extensions of $\mathbb{Q}$ to finite covers of  $X_{\mathbb{Q}}$, with the monodromy of periodic orbits of length $\log p$ under the scaling action corresponding to the Galois action of the Frobenius at the prime p.


\paragraph{Mathematics Subject Classification 2020.}
\href{http://www.ams.org/mathscinet/msc/msc2020.html?t=11R37&btn=Current}{11R37},
\href{http://www.ams.org/mathscinet/msc/msc2020.html?t=11M06&btn=Current}{11M06};
\href{http://www.ams.org/mathscinet/msc/msc2020.html?t=11Mxx&btn=Current}{11M55},
\href{http://www.ams.org/mathscinet/msc/msc2020.html?t=14A15&btn=Current}{14A15},
\href{http://www.ams.org/mathscinet/msc/msc2020.html?t=14F20&btn=Current}{14F20},
\href{http://www.ams.org/mathscinet/msc/msc2020.html?t=57K10&btn=Current}{57K10}.

\end{abstract}

\tableofcontents

\hypersetup{linkcolor=citation}

\section{Introduction}

This paper proposes an extension of class field theory, moving beyond the traditional idele class group description of the abelianized Galois group to a geometric framework of adelic nature. This new perspective establishes a counterpart to Grothendieck’s extension of Galois theory to schemes, focusing on those associated with the ring of integers. This shift not only enriches the conceptual foundation of class field theory but also elucidates the role of the adele class space in  the spectral realization of $L$-function zeros and the explicit formulas of number theory.  

The geometrization of Galois theory, as pioneered by A. Grothendieck in the 1960s, revealed that the \'etale site of an algebraic scheme, combined with its profinite fundamental group, provides an algebro-geometric analogue to classical Galois theory. In contrast, class field theory centers on the idele class group of a global field and its relation, via Artin’s reciprocity map, to the Galois group of the maximal abelian extension. This number-theoretic perspective already suggests the pivotal role of adelic spaces in capturing the geometric essence of abelian field extensions.\vspace{.03in}

According to C. Chevalley\footnote{C. Chevalley "La Th\'eorie du Corps de Classes" Annals of Math. Vol 41, No. 2 (1940)} 
\begin{quote}L'objet de la th\' eorie du corps de classes est de montrer comment les extensions ab\' eliennes d'un corps de nombres alg\' ebriques $K$ peuvent \^etre d\' etermin\' ees par des \' el\' ements tir\' es de la connaissance de $K$ lui-m\^eme; ou, si l'on veut pr\' esenter les choses en termes dialectiques, comment un corps poss\`ede en soi les \' el\' ements de son propre d\' epassement\footnote{The object of class field theory is to show how the abelian extensions of an algebraic number field  $K$  can be determined by elements derived from the knowledge of  $K$  itself; or, if one wishes to present it in dialectical terms, how a field possesses within itself the elements of its own surpassing.}.
\end{quote}

It is therefore natural to anticipate that Grothendieck's extension of Galois theory to schemes will, in the context of global fields $K$, acquire a geometric counterpart of  adelic nature—specifically, one involving the ring $\mathbb{A}_K$ of adeles of $K$. From the very inception of Galois theory, Galois's concept of a \emph{primitive} equation inherently involves not only an abstract group but also, in a fundamental way, the manner in which this group acts on the set of roots of the equation.\vspace{.02in}

The idele class group of a global field $K$ is defined as the quotient group $C_K := K^\times \backslash {\rm GL}_1(\mathbb{A}_K)$. Since ${\rm GL}_1(\mathbb{A}_K)$ acts on $\mathbb{A}_K$ by multiplication, the most natural space on which $C_K$ acts is the adele class space:  
\[
Y_K := K^\times \backslash \mathbb{A}_K, \quad g(y) := gy, \quad \forall g \in C_K, \, y \in Y_K.
\]  
The key notions and properties of this multiplication action of $C_K$ on $Y_K$ (see \cite{C}) are as follows:\vspace{.05in}  

\textbf{Isotropy Subgroup}: The isotropy subgroup associated with the adeles that vanish exactly at a place $ v $ of $ K $ is $ K_v^\times \subset K^\times \backslash {\rm GL}_1(\mathbb{A}_K) $. This subgroup establishes the canonical connection between local and global class field theory.  \vspace{.03in}

\textbf{Explicit Formulas}: The action of the isotropy subgroup $ K_v^\times $ on the transverse space $ K_v $ of the periodic orbit determines the contribution of the place $ v $ to the Weil explicit formulas. This contribution is specifically captured as the trace of the Schwartz kernel $ k(x, y) = \delta(y - \lambda x) $:
\[
\int k(x, x) \, dx = \int \delta(x - \lambda x) \, dx = \frac{1}{|1 - \lambda|}.
\]  

The adele class space $Y_K$ possesses a subtle and delicate structure primarily due to the ergodicity of the action of $K^\times$ on $ \mathbb{A}_K$ under its Haar measure. Even  when $Y_K$ is viewed naively as a topological space, this complexity remains evident and $Y_K$ appears to be highly singular.

Two theoretical frameworks offer valuable methods for analyzing  such spaces, inspired by insights from the geometric case of foliations: \textbf{noncommutative geometry} and \textbf{topos theory}. Each has distinct merits. Noncommutative geometry leverages the powerful tools of functional analysis and Hilbert space operators, while topos theory operates in close synergy with algebraic geometry and sheaf theory. \vspace{.03in} 

The adele class space $ Y_K $ was initially studied using noncommutative geometry, particularly through the crossed product noncommutative algebra $ C_0(\mathbb{A}_K) \rtimes K^\times $. (\cites{C, CM}).  

A notable finding (\cite{CC4}) shows that for $ K = \mathbb{Q} $, the quotient $ X_\mathbb{Q} $ of the space $ Y_\mathbb{Q} := \mathbb{Q}^\times \backslash \mathbb{A}_\mathbb{Q} $ by the maximal compact subgroup $ \widehat{\mathbb{Z}}^* $ of $ C_\mathbb{Q} $ corresponds to the set of points of a naturally associated topos: the \emph{scaling topos}. This topos is defined as the semidirect product of the half-line $ [0, \infty) $ with the multiplicative action of positive integers.
The adelic space $ X_\mathbb{Q} $, along with its cover  
\[
Y_\mathbb{Q} := \mathbb{Q}^\times \backslash \mathbb{A}_\mathbb{Q} \stackrel{\pi}{\longrightarrow} X_\mathbb{Q} = \mathbb{Q}^\times \backslash \mathbb{A}_\mathbb{Q} / \widehat{\mathbb{Z}}^*
\]
offers a geometric framework for the conjectural connection between primes and knots. This  is realized through the periodic orbits $ C_p $ (with length $ \log p $ under the idelic scaling action) corresponding to rational primes, and their preimages $ \pi^{-1}(C_p) $.  

 \begin{figure}
 \centering
\includegraphics[scale=0.75]{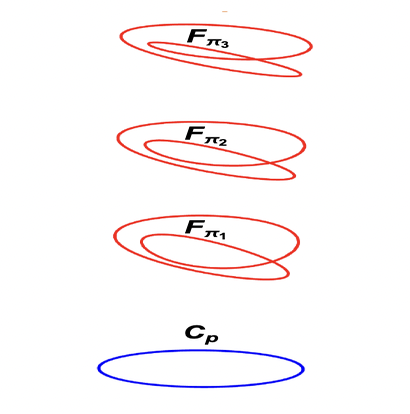}\\
\caption{A typical example of the inverse image of the periodic orbit $C_p$ is a union of components $F_\pi$  where the prime $p$ is unramified and factors as $p=\pi_1\pi_2\pi_3$ in $L$.}\label{fig1}
 \end{figure}

 The main result of this paper is the construction  of a functor that associates finite covers of the space $X_\Q$ to finite abelian extensions  of $ \mathbb{Q} $. For each such extension $\Q\subseteq L $, there is a corresponding open subgroup $ W \subset \widehat{\mathbb{Z}}^* $, defined as the kernel of the natural Galois map $ \widehat{\mathbb{Z}}^* \to \text{Gal}(L/\mathbb{Q}) $, arising from the field inclusions $ \mathbb{Q} \subseteq L \subset \mathbb{Q}^{\text{ab}} $. The associated finite cover of $ X_\mathbb{Q} $ is defined as the double quotient  
$$
X_\mathbb{Q}^L := \mathbb{Q}^\times \backslash \mathbb{A}_\mathbb{Q} / W.
$$  
In section~\ref{sec2} we show that the geometric properties of the finite cover $ \pi^L: X_\mathbb{Q}^L \to X_\mathbb{Q} $ reflect the ramification structure of the extension $ \Q\subseteq L$  and that  the monodromy of the periodic orbits $C_p$ is determined by the Galois action of the arithmetic Frobenius at the prime $p$ over which the cover is unramified. This justifies the assertion that the maximal abelian extension $ \mathbb{Q}^{\text{ab}} $ of the rationals finds its geometric counterpart in the adele class space $ Y_\mathbb{Q}$,  interpreted as the maximal abelian cover  of the adelic space $ X_\mathbb{Q} $.
 \begin{theorem1}\label{mainintro}
  The association $L \rightarrow\left(\pi^L: X_{\Q}^L \to X_{\mathbb{Q}}\right)$ defines a contravariant functor from finite abelian extensions of $\Q$ to finite covers of $X_\Q$. Moreover one has:
  \begin{enumerate}
\item[(i)] The finite set $R$ of places at which the cover ramifies is the union of the archimedean place with the set of primes at which $L$ ramifies.
\item[(ii)] Let $p \notin R$ then the monodromy of $C_p$ in $X_{\mathbb{Q}}^L$ is the element of $\text{Gal}(L/\mathbb{Q})$ given by the Frobenius $\mathrm{Frob}_p$.
\item[(iii)] The connected components of the inverse image of $C_p$ are circles labeled by the places of $L$ over the prime $p$. 
\end{enumerate}
 \end{theorem1}
 Furthermore, the "classically visible" geometry of $ X_\mathbb{Q} $ is captured by the union $ \mathcal{X}_\mathbb{Q} $ which includes all periodic orbits $ C_p $ (as well as a single point corresponding to the archimedean place) and the generic orbit $ \eta \simeq \mathbb{R}_+^* $, under the scaling action of $ \mathbb{R}_+^* $. The topology of $ \mathcal{X}_\mathbb{Q} $ corresponds to that of a covering of $ \overline{\text{Spec} \, \mathbb{Z}} $, where each place $ v $ is covered as follows:\vspace{.02in}  

- For a rational prime $ p $, the place $ v $ is covered by the periodic orbit $ C_p $, which forms a circle of length $ \log p $.  \vspace{.01in}

- For the archimedean place $ v = \infty $, it is covered by a single point.\vspace{.01in}

- The generic point of $ \text{Spec} \, \mathbb{Z} $ is covered by the generic orbit $ \eta\subset X_\Q$.  \vspace{.03in}

The density of the generic point in $ \text{Spec} \, \mathbb{Z} $ is mirrored by the density of the generic orbit $ \eta $ in $ X_\mathbb{Q} $, which in turn reflects the density of ideles within adeles. More precisely the topology of the space of orbits of the scaling action on $ \mathcal{X}_\mathbb{Q} $ coincides with the topology of $ \overline{\text{Spec} \, \mathbb{Z}} $.\vspace{.02in} 

The covering $ \mathcal{X}_\mathbb{Q} $ of $ \overline{\text{Spec} \, \mathbb{Z}} $, together with the natural scaling action of $ \mathbb{R}_+^* $, can be interpreted as the class field theory counterpart of the points of a genuine  curve $\mathscr C$  over  the analogue of the algebraic closure of a finite field, equipped with the action of the Frobenius automorphisms. This class field theory counterpart to classical algebraic geometry over finite fields provides valuable insights into the essence of the geometry of $ \mathscr{C} $ and the nature of the coefficients involved in defining its points. The Riemann-Roch theorem for the ring $\Z$ proven in \cite{CC6,CC7}, along with the interpretation of $\mathbb{Z}$ as a ring of polynomials  over the combinatorial realization $\mathbb{S}$ of the sphere spectrum, suggests heuristically that $\mathscr{C}$ resembles a one-dimensional projective space $\mathbf{P}^1$.

Interpreting $\mathcal{X}_\mathbb{Q}$ as the class field theory counterpart of the curve $\mathscr C$  suggests  intriguing insights. For instance, the unique periodic orbit consisting of a single point is associated with the archimedean place  thus the point $\infty$ appears to be the only one defined over $\mathbb{S}$. The following table presents an excerpt from this proposed interpretation

\begin{center}
\renewcommand{\arraystretch}{0.5}
\begin{tabular}{c|c|c}
& &\\ Function field $K=\F_q(t)$ & Global field $\Q$ & Class field  of $\Q$\\
 & & \\
\hline 
& &\\ $\overline{\text{Spec}\,\F_q[t]}$  &  $\overline{\text{Spec}\,\mathbb{Z}}$ & $X_\Q$\\
 & &\\
& &\\ Points of~   $\mathbf P^1\otimes_{\F_q}\bar\F_q$ & The curve $\mathscr C$ &$\mathcal X_\Q$ \\
& & \\
& &\\ $\mathrm{Frob}_p$ on the points  & Action of Weil group& Scaling action of $\R_+^*$ \\ & & \\
& &\\ $\mathrm{Frob}_p$ orbits   & Places of $\Q$ &Scaling  orbits $C_p$  \\   & &\\
& &\\ Finite abelian     & Finite abelian &Finite abelian cover    \\cover  of $\mathbf P^1$ & extension of $\Q$ & $\mathcal X_\Q^L$ of $\mathcal X_\Q$\\& & \\ 
& &\\ Pro-Abelian   & Maximal abelian &$ \left(\mathcal Y_\Q=\varprojlim_L \mathcal X_\Q^L\right)\to \mathcal X_\Q$   \\   cover $\widehat{\mathbf P}^1_{\text{ab}}$  & extension $\Q^{ab}$ &\\& & \\ 
\end{tabular}
\end{center}
 
The next diagram highlights the parallel between the perspective in classical algebraic geometry and that in adelic geometry related to the maximal abelian extension of  global fields:

  \[
 \begin{tikzcd}
     \widehat{\mathbf P}^1_{\text{ab}}\arrow[d,"\text{pro-ab}"'] & & \Q^{ab}\arrow[d,dash,"\text{Iwasawa}"',,"{\widehat\Z^*}"] & & \mathcal Y_{\Q}\arrow[d,"{\widehat\Z^*}"']\\
     \mathbf P^1\otimes_{\mathbb F_q}\overline{\mathbb F}_q & & \Q & & \mathcal X_\Q
 \end{tikzcd} 
 \]

We now exploit the Mumford-Mazur analogy between knots and primes to investigate the mutual linking of the periodic orbits $ C_p$.  
The linking number $ \left(C_1, C_2\right)$ of two disjoint oriented knots in  $S^3$ is determined using the natural morphism $\pi_1(C_1) \to \pi_1(S^3 - C_2)^{\text{ab}}$ and the isomorphism $\pi_1(S^3 - C_2)^{\text{ab}} \cong \mathbb{Z}$, which is derived from Alexander-Pontryagin duality combined with the Hurewicz isomorphism.  
The Mumford-Mazur analogy between knots and rational primes is summarized in the following table:

\begin{center}
\renewcommand{\arraystretch}{0.5}
 \begin{tabular}{c | c}
& \\ \textbf{Knot} $C$ & \textbf{Prime} $p$ \\
& \\
\hline 
& \\ Inclusion $C \subset S^3$ & $r^*: \operatorname{Spec} \mathbb{F}_p \hookrightarrow \operatorname{Spec} \mathbb{Z}$ \\ & \\
& \\ Knot complement $X=S^3-C$ & $\operatorname{Spec} \mathbb{Z} \backslash\{p\}$ \\ & \\
& \\$\pi_1\left(S^3-C\right)^{a b}=\mathbb{Z}$ & $\pi_1^{e t}(\operatorname{Spec} \mathbb{Z}[\frac 1p])^{a b}=\mathbb{Z}_p^*$ \\
& \\
& \\ $\pi_1\left(C_1\right) \rightarrow \pi_1\left(S^3-C_2\right)^{a b}$ & $\pi_1^{e t}\left(\operatorname{Spec} \mathbb{F}_p\right) \rightarrow \pi_1^{e t}(\operatorname{Spec} \mathbb{Z}[\frac 1q])^{a b}$ \\
& \\
& \\ Linking Number $\left(C_1, C_2\right)$ & $p \in \mathbb{Z}_q^*$ \\
& \\
\end{tabular}\end{center}

For a given rational prime $ p $, the linking with all other (rational) primes $ q $ is understood at once as the natural map $r^*:\pi_1^{e t}\left(\operatorname{Spec} \mathbb{F}_p\right) \rightarrow \pi_1^{e t}(\operatorname{Spec} \mathbb{Z}_{(p)})^{a b}$. The removal of all primes except $ p $ from $ \text{Spec} \, \mathbb{Z} $ is effected algebraically by replacing $ \text{Spec} \, \mathbb{Z} $ with the spectrum of the local ring $ \mathbb{Z}_{(p)} $. In our set-up, primes are represented by the periodic orbits $ C_p $, and their linking is revealed explicitly through the following theorem

\begin{theorem2}\label{main2intro}
	 Let  $p$ be a rational prime. Let ${\rm F r o b}_{p}\in \pi_1^{e t}(\Spec(\F_p))$ be the canonical generator. 
     \begin{enumerate}
         \item[(i)] The inverse image $\pi^{-1}(C_p)\subset Y_\Q$ of the periodic orbit $C_p$  is canonically isomorphic to the mapping torus of the multiplication by $r^*\left\{{\rm F r o b}_{p}\right\}$  in the abelianized \'etale fundamental group $\pi_1^{e t}(\Spec \, \Z_{(p)})^{ab}$.
          \item[(ii)] The canonical isomorphism in $(\rm i)$ is equivariant for the action of the idele class group.
         \item[(iii)] The monodromy of the periodic $C_p$ in $\pi^{-1}(C_p)\subset Y_\Q$ is equal to the natural map $r^*:\pi_1^{e t}\left(\operatorname{Spec} \mathbb{F}_p\right) \rightarrow \pi_1^{e t}(\operatorname{Spec} \mathbb{Z}_{(p)})^{a b}$, and determines the linking of the prime $p$ with all other primes.
         \end{enumerate}
     \end{theorem2}

 \begin{figure}[H]
 \centering
\includegraphics[scale=0.3]{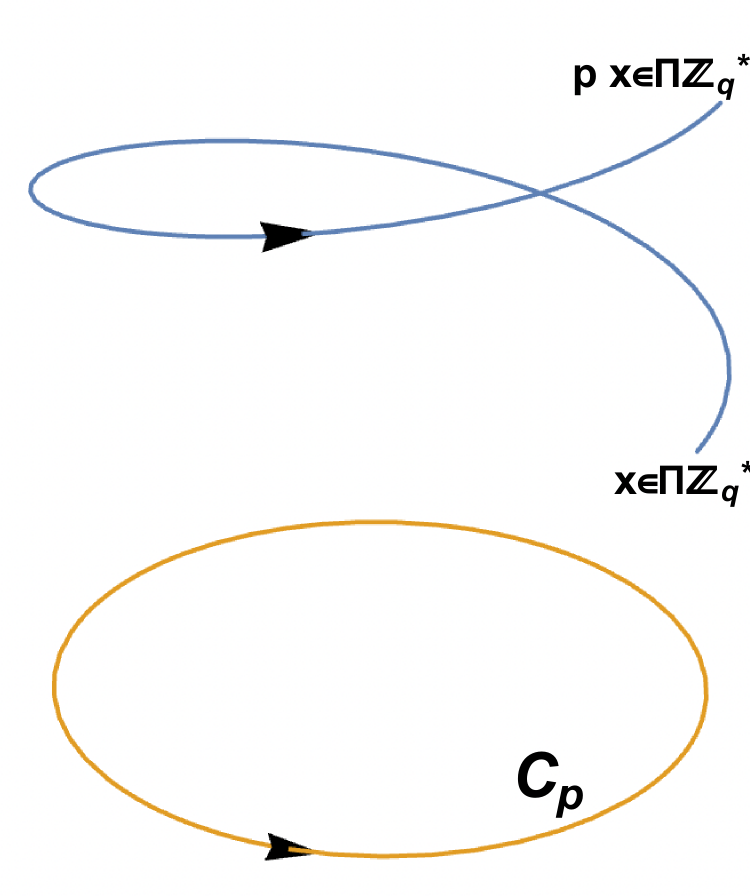}\\
\caption{The monodromy of the periodic $C_p$ in $\pi^{-1}(C_p)\subset Y_\Q$}\label{monodrom}
 \end{figure}

In section~\ref{sec3} we provide a proof to the fact that the abelianized étale fundamental group $ \pi_1^{\text{et}}(\text{Spec}\,\mathbb{Z}_{(p)})^{\text{ab}} $ is isomorphic to the product $ \prod_{q \neq p} \mathbb{Z}_q^* $. The image $ r^*(\text{Frob}_p) \in \pi_1^{\text{et}}(\text{Spec}\,\mathbb{Z}_{(p)})^{\text{ab}} $ corresponds to the element $ p $ diagonally embedded in the product $ \prod_{q \neq p} \mathbb{Z}_q^* $.  
The injectivity of the map $\mathbb{Z}\ni n \mapsto p^n \in \mathbb{Z}_q^* $ for any $ q \neq p $ highlights the non-trivial linking of the periodic orbits $ C_p $ and $ C_q $.  \vspace{.02in}

The fundamental properties underpinning the analogy between $ \overline{\text{Spec}\,\mathbb{Z}}$ and the 3-sphere $ S^3 $ are summarized in the following table:\vspace{.08in}

\begin{center}
\renewcommand{\arraystretch}{0.5}
\begin{tabular}{c |c}
&\\ \textbf{Sphere} $S^3$ &  $\overline{\text{Spec}\mathbb{Z}}$ \\
 & \\
\hline 
& \\$\pi_1\left(S^3\right)=\{1\}$ & $\pi_1^{e t}(\overline{\operatorname{Spec} \mathbb{Z}})=\{1\}$ \\
& \\
& \\$H^3\left(S^3, \mathbb{Z}\right)=\mathbb{Z}$ & $H^3\left(\overline{\operatorname{Spec} \mathbb{Z}}, \mathbf{G}_m\right)=\mathbb{Q} / \mathbb{Z}$ \\ & \\
\end{tabular}
\end{center}
 \vspace{0.4cm}

 Let $j: \eta \rightarrow \operatorname{Spec} \mathbb{Z}$ be the generic point. A crucial element in the proof of the Artin-Verdier Poincar\'e duality is the exact sequence of sheaves on $\Spec \Z$
 $$
0 \rightarrow \mathbf{G}_m \rightarrow j_* \mathbf{G}_{m, \eta} \rightarrow \coprod_p \mathbb{Z}{\mid\!p} \rightarrow 0
$$
Here, $\mathbf{G}_m$ denotes the representable sheaf giving the multiplicative group, and the last term on the right is the sheaf of Cartier divisors. 
One then utilizes 
the long exact sequence of cohomology groups
$$
\begin{gathered}
0 \rightarrow \mathrm{H}^2\left(\Spec \Z, \mathbf{G}_{m}\right) \rightarrow \mathrm{H}^2\left(\eta, \mathbf{G}_{m, \eta}\right) \stackrel{r_1}{\rightarrow} \bigoplus_p \Q / \Z\rightarrow \\ \rightarrow \mathrm{H}^3\left(\Spec \Z, \mathbf{G}_{m}\right) \rightarrow \mathrm{H}^3\left(\eta, \mathbf{G}_{m, \eta}\right) \rightarrow 0.\end{gathered}
$$
In this sequence,  $\mathrm{H}^3(\Spec \Z, \mathbf{G}_{m})$ appears as the cokernel of the map $r_1$, which, as established by the Brauer-Hasse-Noether theorem, is a single copy of $\Q/\Z$. Due to the presence of the real archimedean place with Brauer group $\Z/2\Z$, there is a subtlety when dealing with higher cohomology groups. 
 
 Two aspects of this development are particularly relevant to our context, and we will now address them. The first is the pivotal role of the inclusion $j: \eta \to \text{Spec} \, \mathbb{Z}$ of the generic point, which suggests a hidden role of the  relative cohomology for the pair $ (\text{Spec} \, \mathbb{Z}, \eta)$.  In the adelic framework, a similar structure emerges. For instance, the spectral realization of zeros of $L$-functions can be interpreted as the action of the idele class group on $ H^1(Y_\mathbb{Q}, \eta) $, where the generic orbit $ \eta $ (i.e., the idele classes) takes on the role of the generic point \cite{CM}, Theorem 4.116.  
A second aspect of the proof of Artin-Verdier duality that is relevant in our context is its reliance solely on abelian extensions of $ \mathbb{Q} $. Indeed, O. Gabber demonstrated\footnote{private communication} that the duality continues to hold when the \'etale site is replaced by the site restricted to abelian extensions. 

The considerations above suggest a promising direction of development using topos theory and the analogue of the "abelian  site" of $ \text{Spec} \, \mathbb{Z} $ over $ X_\mathbb{Q} $. Furthermore, the intricate topos-theoretic structure of the scaling topos, together with its structure sheaf, offers the potential to intrinsically understand the topos counterpart of the finite abelian covers of $ X_\mathbb{Q} $.  

The scaling topos serves as the topos viewpoint of the adelic space $ X_\mathbb{Q} $ and, as shown in \cite{CC4}, its interpretation via an extension of scalars from the arithmetic site (see \cite{CC3}) equips the scaling topos with a natural structure sheaf. The sections of this sheaf are piecewise affine, continuous, convex functions. A remarkable consequence of this additional structure is observed in the periodic orbits $ C_p $: each $ C_p $, when equipped with the restriction of the scaling topos’s structure sheaf, becomes an analogue of an elliptic curve within the framework of characteristic one.\vspace{.02in}

Noncommutative geometry provides powerful techniques for understanding the complex geometry of the adelic spaces $ Y_\mathbb{Q} $ and $ X_\mathbb{Q} $. This approach is particularly appealing because it equips us with the necessary tools for the spectral realization of $L$-functions. A clearer understanding of the adelic spaces $ Y_\mathbb{Q} $ and $ X_\mathbb{Q} $ can be gained by considering their "semilocal" versions, which are obtained by restricting the places of $ \mathbb{Q} $ to a finite subset that includes the archimedean place.  
Following the presentation of the semilocal setup in Section \ref{semilocprep}, we demonstrate in Section \ref{sheafsemiloc} that the semilocal algebras constitute a sheaf of algebras over $\Spec\,\Z$. The semilocal algebras $\cS(\A_S)\rtimes \Z_S^\times$ are constructed as the cross products of the Bruhat-Schwartz algebras $\cS(\A_S)$ of functions on semilocal adeles, by the multiplicative groups given by $\Z_S^\times=\mathbf{G}_m(S^c)$, sections of the sheaf $\mathbf{G}_m$ on the open set complement of $S$ in $\Spec\,\Z$. We first show (Proposition \ref{structure}) that the semilocal Bruhat-Schwartz algebra $\cS(\A_S)$ form a sheaf $\sheaf$ of commutative algebras over $\Spec\,\Z$. We then obtain the following result which establishes the compatibility of the noncommutative geometric constructions with the algebraic geometry of $\Spec\,\Z$. 
\begin{theorem3}	
\label{structure1intro}  
 		\begin{enumerate}
 			\item The algebraic cross product $\sheaf\rtimes \mathbf{G}_m$ defines a sheaf  of algebras on  $\Spec\Z$ such that for every finite set of places $S\ni \infty$
 			$$
 			\left(\sheaf\rtimes \mathbf{G}_m\right)(S^c)=\cS(\A_S)\rtimes \Z_S^\times 			$$
 			\item The stalk of $\sheaf\rtimes \mathbf{G}_m$ at the generic point $\eta$ is the global cross product $\cS(\A_\Q)\rtimes \Q^\times$. 			\item The global sections of $\sheaf\rtimes \mathbf{G}_m$ form the cross product $\cS(\R)\rtimes \{\pm 1\}$.
 		\end{enumerate}
\end{theorem3}

In noncommutative geometry, a fundamental method for handling complex quotients is the interplay between the homotopy-theoretic treatment of such spaces and the cross-product $ C^*$-algebras that encode their topology. The central tool in this context is the assembly map  underlying the Baum-Connes conjecture. This conjecture, proven for cross-products by abelian groups \cite{HK}, is sufficient for computing the $K$-theory of the $C^*$-algebras relevant in our context, and was used in \cite{RM} to compute the $K$-theory of the  cross product $C^*$-algebra associated to the adele class space of global fields. 

In section \ref{sect4},  we investigate the $ K $-theory of the "visible part" (a precise definition is given there) of $ X_\mathbb{Q} $, localized to a finite set of places. Specifically, for concreteness we treat the case $S=\{p,q,\infty\}$, and in that case we compute the $ K $-theory of the $ C^* $-algebra $ A $ corresponding to the union, within the semi-local space $ X_{\mathbb{Q}, S} $, of the generic orbit and the three periodic orbits $ C_p, C_q, $ and $ C_\infty $.  
The resulting isomorphism $ K_0(A) \simeq \mathbb{Z}^3 $ captures the presence of the three periodic orbits, while $ K_1(A) \simeq \mathbb{Z}^2 $ reflects the one-dimensional nature of the periodic orbits $ C_p $ and $ C_q $.

Finally, section \ref{sect5} provides a general description of the homotopy quotient $ \mathcal{Y}_S $, which offers a homotopy-theoretic approach to the semilocal adele class space $ Y_S $. The space $ \mathcal{Y}_S $ is locally compact, with dimension equal to the cardinality $ \#S $. It is equipped with an action of $ \mathbb{R}^n $, where $ n = \#S - 1 $, by translations. This action defines a codimension 1 lamination, with the space of leaves being $ Y_S $.

   \section{Notations}
 Let $K$ be a global field. We denote by $K^\times=K\setminus\{0\}$ the multiplicative group of non-zero elements of $K$.
  
 The adele ring of a global field $K$ (in particular when $K=\mathbb Q$) is the restricted product $\mathbb A_K:=\prod (K_v,\mathcal O_v)\subset \prod K_v$ consisting of the tuples $(a_v)$ where $a_v$ lies in the subring $\mathcal{O}_v \subset K_v$ for all but finitely many places $v$. The index $v$ ranges in the set $\Sigma_K$ of all places of the global field $K$, 
 $K_v$ is the local field completion of $K$ at the place $v$, and $\mathcal{O}_v$ is the corresponding valuation ring when the place $v$ is non-archimedean.

We denote the adele class space of $ \mathbb{Q} $ by $ Y_\mathbb{Q} = \mathbb{Q}^\times \backslash \mathbb{A}_\mathbb{Q} $. The action of $ \mathbb{Q}^\times $ on the rational adeles is defined through its image under the natural diagonal embedding $ \mathbb{Q} \hookrightarrow \mathbb{A}_\mathbb{Q} $, which defines the principal adeles. Throughout this paper we denote by $X_\Q$ the natural quotient  $\mathbb{Q}^\times \backslash \mathbb{A}_\mathbb{Q}/\widehat\Z^*$, where $\widehat\Z^*=\prod_{v<\infty}\Z_v^*$ is the subgroup of $\A_\Q$ of all $p$-adic units.
  
We recall from \cite{CC4} that the space of points (up-to isomorphism) of the sector $X_\Q=\mathbb{Q}^\times \backslash \mathbb{A}_\mathbb{Q}/\widehat\Z^*$ of the adele class space $Y_\Q$ is in canonical bijection with the space of points of the scaling topos $\mathscr S$. This is the Grothendieck topos $[0,\infty)\rtimes \mathbb N^\times$ of $\N^\times$-equivariant sheaves of sets on the real interval $[0,\infty)$. The scaling topos is  canonically isomorphic to the category $\mathfrak{Sh}(\mathcal C,J)$ of sheaves of sets on the site $(\mathcal C,J)$. Here, $\mathcal C$ denotes the small category of bounded open (possibly empty) intervals $\Omega\subset [0,\infty)$ including those of the form $[0,a)$ for $a>0$. The multiplicative monoid $\N^\times$ acts on the objects of $\mathcal C$ and defines the morphisms between two objects as $\operatorname{Hom}_{\mathcal C}(\Omega,\Omega') =\{n\in\N^\times|n\Omega\subset\Omega'\}$. The category $\mathcal C$ is endowed with the Grothendieck topology $J$ generated by ordinary coverings of the topological space $[0,\infty)$: we refer to \cite{CC4} for more details.

The points of the adelic space $X_\Q$ (and also of the  scaling topos) are in canonical bijection with abstract rank one groups and rank one subgroups of $\R$.

Let $p$ be a rational prime. The collection of points of  $X_\Q$ corresponding to subgroups $H\subset \mathbb R$ which are abstractly isomorphic to the subgroup $H_p\subset\Q$ of fractions with denominators a power of $p$ determine a subspace $C_p\subset X_\Q$ called the ($p$-)periodic orbit. The map $\mathbb{R}_{+}^* \rightarrow C_p$, $\lambda \mapsto \lambda H_p$ induces an isomorphism  $\mathbb{R}_{+}^*/ p^{\mathbb{Z}} \stackrel{\sim}{\rightarrow} C_p$ showing that $C_p$ is a circle of length $\log p$. We refer to \cite{CC4} for more details.

Finally, we briefly recall the notion of Bruhat-Schwartz function \cite{Br}. 
On a real vector space the Bruhat-Schwartz functions are the smooth functions all of whose derivatives are rapidly decreasing. On a non-archimedean local field $K$, a Bruhat-Schwartz function is a locally constant function of compact support. \newline The space of Bruhat-Schwartz functions on the adeles $\mathbb{A}_K$ of a global field $K$ is defined to be the restricted tensor product $\bigotimes'_v\mathcal{S}(K_v)=\varinjlim_E\left(\bigotimes_{v\in E}\mathcal S(K_v)\right)$
 of Bruhat-Schwartz spaces $\mathcal{S}(K_v)$ of local fields, where $E$ is a finite set of places of $K$. 

\section{Abelian extensions of $\Q$ and finite covers of the adelic space $X_\Q$}\label{sec2}

In this section, we  construct a natural functor which associates to a finite abelian extension of the field of rational numbers a covering of $X_\Q$ with finite fibers. \newline We first define a covering of $X_\Q$ as follows
\begin{defn}\label{coverdef}Let $G$ be a finite abelian group and $\chi :\widehat\Z^*\to G$ be a surjective continuous morphism. We let 
\begin{equation}\label{ccdef}
	X_\Q^\chi:=Y_\Q\times_{\widehat\Z^*}G, \  \  \pi_\chi: X_\Q^\chi\to X_\Q
\end{equation}
be the quotient of the product $Y_\Q\times G$ by the action of $\widehat\Z^*$ and $\pi_\chi$ be the first projection on $Y_\Q/\widehat\Z^*=X_\Q$.
\end{defn}
For $x\in X_\Q=\Q^\times\backslash \A_\Q/\widehat\Z^*$ in the class of the adele $a=(a_v)\in \A_\Q$, the set of places
\begin{equation}\label{zxdef}
Z(x)=\{v\in \Sigma_\Q\mid a_v=0\}, \  \ 
\end{equation}
is independent of the choice of $a$ since the elements of $\Q^\times \times \widehat\Z^*$ are ideles.\newline
The group $G$ acts on $X_\Q^\chi$ since the action of $k\in G$ on $Y_\Q\times G$ by multiplication $(z,g)\mapsto (z,gk)$ is compatible with the quotient by $\widehat\Z^*$.  Indeed for $(z',g')\sim (z,g)$ one has $w\in \widehat\Z^*$ with $(z',g')=(zw,\chi(w)g)$ and
$$ (z',g'k)= (zw,(\chi(w)g)k)=(zw,\chi(w)(gk))\sim (z,gk)$$
The nature of the action of $G$ is specified by the next lemma.
\begin{lem}\label{Gacts}Let $G$, $\chi$ and $\pi_\chi$ be as in Definition \ref{coverdef} and $x\in X_\Q$.\newline  $(i)$~The group $G$ acts transitively on the fiber $F_x=\pi_\chi^{-1}(x)$.\newline
$(ii)$~Assume  $\infty \notin Z(x)$, then the stabilizer of any  $y\in F_x$ is 
\begin{equation}\label{isot}
\chi\left(H\right)\subset G, \  \ H=\prod_{v\in Z(x)}\Z_v^*\subset \widehat\Z^*.
\end{equation}
\end{lem}
\proof $(i)$~Let $x\in X_\Q=Y_\Q/\widehat\Z^*$ be in the class of $y\in 
Y_\Q$. Let $(z,g)\in Y_\Q\times G$ such that $ \pi_\chi(z,g)=x$. There exists $w\in \widehat\Z^*$ such that $wz=y$ and one has, modulo the action of $\widehat\Z^*$, $(z,g)\sim (y,\chi(w)g)$. Thus every element of  $F_x=\pi_\chi^{-1}(x)$ is equivalent to a pair $(y,k)$ for some $k\in G$ and the action of $G$ on such pairs is transitive.\newline
$(ii)$~Let $(z,g)\in Y_\Q\times G$ represent $y\in \pi_\chi^{-1}(x)$. Let $k\in G$ such that $yk=y$. One then has  $(z,g)\sim (z,gk)$ modulo  the action of $\widehat\Z^*$. Thus there exists  $w\in \widehat\Z^*$ with 
$$
zw=z, \ \ \chi(w)g=gk
$$
Let $a\in \A_\Q$ in the class of $z\in \Q^\times\backslash\A_\Q=Y_\Q$. The equality $zw=z$ means that there exists $q\in \Q^\times$ such that $aw=qa$. Since the archimedean component of $a$ is non-zero this implies that $q=1$, and that for any finite prime $p$ one has $a_pw_p=a_p$. Thus it holds if and only if $w\in H=\prod_{v\in Z(x)}\Z_v^*\subset \widehat\Z^*$ and since $k=\chi(w)$ one obtains $(ii)$.\endproof 
\subsection{The covering of the periodic orbits $\pi:\pi^{-1}(C_p)\to C_p$}\label{sect2.1}

This part is devoted to demonstrating the non-triviality of the covering restricted on the periodic orbits $C_p$. \newline   
To understand the role of the quotient by $\mathbb{Q}^\times$ in the above constructions, we first examine the situation prior to applying this quotient. 
\begin{lem}\label{trivia} The quotient map $p:\A_\Q\to \A_\Q/\widehat\Z^*$ admits a continuous  section 
$$\rho:\A_\Q/\widehat\Z^*\to \A_\Q,\quad  p\circ \rho(x)=x, \  \forall x\in \A_\Q/\widehat\Z^*.$$	
\end{lem}
\proof Let $a=(a_v)\in \A_\Q$.  For each finite place $p$, there exists an element $w_p\in \Z_p^*$ such that $a_p=w_p \rho_p$ where $\rho_p=0$ if $a_p=0$ and is a power of $p$, $\rho_p= p^n$, when $a_p\neq 0$. Both $w_p$ and the power $n$ are uniquely determined if $a_p\neq 0$. One lets $\rho_\infty=a_\infty$ and one has the equality 
$$
a=w\rho, \  \ w_v= w_p, \ \forall v=p<\infty,\  w_\infty =1, \ \rho_v=\rho_p, \ \forall v=p<\infty, \rho_\infty =a_\infty
$$
The element $\rho(a)\in \A_\Q$ is uniquely defined, only depends on the class of $a$ in $\A_\Q/\widehat\Z^*$ and belongs to this class. The map $\rho$ is continuous since each component $\rho_v(a)$ is a  continuous function of $a_v$. Thus $\rho$ gives the required section. \endproof 
Next, we consider the division by $\Q^\times$. The cross-section $\rho$  constructed above  is not equivariant under the action of $\Q^\times$. This action, which is the diagonal action on adeles, modifies  each component $\rho(x)_p$ of 
$\rho(x)$ by its product $q\rho(x)_p$. Specifically, for a prime $q$, this results in replacing $\rho(x)_p$ with $q\rho(x)_p$,  which is not a power of $p$ when $p\neq q$. \newline
To demonstrate that the fibration $\pi:\Q^\times\backslash\A_\Q=Y_\Q\to X_\Q$  is not trivial we examine  its restriction to a periodic orbit $C_p$. 

\begin{prop}\label{mappingtorus} Let $p$ be a prime and $H=\prod_{v\neq p}\Z_v^*$ with $v<\infty$.

$(i)$~The map $\eta:H\times \R_+^*\to \A_\Q$ defined  componentwise by
$$
\eta((h,\lambda))_v=h_v, \forall v<\infty, \ v\neq p, \ \ \eta((h,\lambda))_p=0, \ \ \eta((h,\lambda))_\infty =\lambda
$$
induces a canonical isomorphism $\tilde \eta$ of the mapping torus of the homeomorphism of the compact group $H$ given by multiplication by $p\in H$ and the inverse image $\pi^{-1}(C_p)$. \newline
$(ii)$~The canonical isomorphism $\tilde \eta$ as in $(i)$, is equivariant for the action of the idele class group.\newline
$(iii)$~The covering $\pi:\pi^{-1}(C_p)\to C_p$ has a non-trivial monodromy given by multiplication by $p$ in $H$.	
\end{prop}
\proof $(i)$~We first show that every element of $\pi^{-1}(C_p)$ is in the range of $\tilde \eta$.
The elements of $C_p$ are the double classes of adeles of the form
\begin{equation}\label{adlam}
a(\lambda), \  \ a(\lambda)_\infty= \lambda >0, \ a(\lambda)_p=0, \ \ a(\lambda)_v=1, \forall v<\infty, \ v\neq p	
\end{equation}
 The fiber $\pi^{-1}(C_p)\subset \Q^\times\backslash\A_\Q$ is formed of adeles $\alpha=(\alpha_v)$ such that there exists $q$ in $\Q^\times$, $u \in  \widehat\Z^*$ and $\lambda>0$ such that $\alpha= q a(\lambda)u$, \ie 
 \begin{equation}\label{adlam1}
\alpha_v=q u_v, \forall v<\infty, \ v\neq p, \ \ \alpha_p=0, \ \alpha_\infty =q \lambda.
\end{equation}
Let $q=q' \times p^k$ where $q'$ is prime to $p$. Then 
   modulo $\Q^\times$,  $\alpha$ is equivalent to $\alpha'=q'^{-1}\alpha$ and one  has
    \begin{equation}\label{adlam2}
\alpha'_v\in \Z_v^*,\  \forall v<\infty, \ v\neq p, \ \ \alpha'_p=0, \ \alpha'_\infty =p^k \lambda
\end{equation}
  This shows that $\alpha$ belongs to the range of   $\tilde \eta$. Moreover two elements  $\alpha, \beta\in H\times \R_+^*$ fulfill $\tilde \eta(\alpha)=\tilde \eta(\beta)$ if and only if they are in the same class modulo $\Q^\times$, \ie there exists $q\in \Q^\times$ with $\eta(\beta)=q\eta(\alpha)$. This implies, using \eqref{adlam}, that $q>0$ and that $\vert q\vert_v=1$ for $v<\infty, \ v\neq p$ so that $q=p^n$ for some $n\in \Z$. Thus $\tilde \eta$ 
  \begin{equation}\label{eta}
  \tilde \eta: p^\Z\backslash \left(H\times \R_+^*\right)\to \pi^{-1}(C_p)
  \end{equation}
  is an isomorphism. In general, one constructs the mapping torus of an homeomorphism $\phi:Y\to Y$ as the flow given by translations of the $\R$ variable in the quotient
 $(Y\times \R)/\Z$ where the generator of $\Z$ acts by the map 
 $$
 \psi(y,t):=(\phi(y),t+1).
 $$
 Taking $Y=H$, $\phi(y)=py$ and using the logarithm  $\log:\R_+^*\to \R$, the scaling action of $\R_+^*$ on $p^\Z\backslash \left(H\times \R_+^*\right)$ becomes identical to the flow associated to the homeomorphism of the compact group $H=\prod_{v\neq p}\Z_v^*$ given by multiplication by $p\in H$. By using the logarithm with base $p$ one obtains the translation $t\mapsto t+1$.\newline
 $(ii)$~The idele class group $C_\Q$ is lifted in ideles as the subgroup ${\widehat \Z}^*\times \R_+^*$. It acts on $H\times \R_+^*$ by multiplication componentwise and similarly on adeles. This shows that $\tilde \eta$ is equivariant for the action of $C_\Q$ on both sides. \newline
$(iii)$~The map $\pi:\pi^{-1}(C_p)\to C_p$ is a fibration,  whose monodromy can be identified with an element of the automorphism group \text{Aut}(F), describing how the fiber is transformed after traversing a loop in the base space. From $(i)$, \ie  the identification of $\pi^{-1}(C_p)$ with the mapping torus, it follows that  this monodromy is given by the multiplication by $p$ in $H$. Furthermore for each prime $q\neq p$  the map $\Z\to \Z_q^*$ given by $n\mapsto p^n$ is injective since $\Q$ injects in $\Q_q$. A fortiori the map $n\mapsto p^n\in H$ is injective. Thus  this monodromy is non-trivial and the map $\pi:\pi^{-1}(C_p)\to C_p$ does not admit a continuous cross section.  \endproof 

\subsection{Ramification of the cover $\pi_\chi: X_\Q^\chi\to X_\Q$}
We now return to the finite cover $\pi_\chi$ from Definition \ref{coverdef} to examine its “ramification”. Let $x\in X_\Q$ with $\infty\notin Z(x)$. Lemma \ref{Gacts} establishes that the group $G$ acts transitively on the fiber $F_x=\pi_\chi^{-1}(x)$, with isotropy described by \eqref{isot}. Consequently, one has
\begin{equation}\label{rami0}
G \  \textit{acts freely transitively on} \ F_x \iff \chi(\prod_{v\in Z(x)}\Z_v^*)={1}.
\end{equation}
In cases where  condition \eqref{rami0} is satisfied, the fiber $F_x$ with its $G$-action behaves like a principal bundle. Otherwise, “ramification” occurs. Let us introduce its definition
\begin{defn}\label{deframi} Let $S$ be a finite set of places. By definition, the morphism $X_\Q^\chi\stackrel{\pi_\chi}{\to}X_\Q$ is said to be {\em unramified outside $S$} if and only if  the following condition holds for every $x\in X_\Q$:
$$
Z(x)\cap S=\emptyset~ \Longrightarrow \ G \  \textit{acts freely transitively on} \ F_x
$$	
\end{defn}
Since the morphism $\chi:\widehat\Z^*\to G$ is continuous and $G$ is finite, the kernel $\ker \chi$ is an open subgroup of $\widehat\Z^*$ and hence contains a neighborhhod of $1$. The topology  of $\widehat\Z^*=\prod \Z_p^*$ is the product topology of the infinite product, and a basis of neighborhoods of $1$ is formed of product open sets where $C$ is a finite set of primes and the $W_p$'s, with $p\in C$, are open in $\Z_p^*$
$$
W=\prod_{p\in C} W_p \times \prod_{q\notin C} \Z_q^*.
$$
Thus there exists a finite set $C$ of primes such that 
$
\chi(\prod_{q\notin C} \Z_q^*)=\{1\}
$.
Let
\begin{equation}\label{cchi}
C(\chi):=\{p\in\Sigma_\Q\setminus\{\infty\}\mid \chi(\Z_p^*)\neq \{1\}\}
\end{equation}
then the set $C(\chi)$ is finite and $\chi(\prod_{q\notin C(\chi)} \Z_q^*)=\{1\}$ by continuity.
In other words, the finite set $C(\chi)$ corresponds to the set of divisors of the {\em conductor} of $\chi$. The conductor is defined, using the description of $\widehat\Z^*$ as the projective limit of the finite groups $(\Z/m\Z)^*$, as the gcd of all integers $m$ for which $\chi$  factors through the canonical projection $\widehat\Z^*\to (\Z/m\Z)^*$.\newline
The following proposition determines the finite sets of places $S\ni \infty$  outside which $\pi_\chi$ is unramified. 
\begin{prop}\label{ramiprop} Let $X_\Q^\chi\stackrel{\pi_\chi}{\to}X_\Q$ be as in Proposition \ref{coverdef}. Then for a finite set of places $S\ni \infty$ one has 
$$
X_\Q^\chi\stackrel{\pi_\chi}{\longrightarrow}X_\Q \  \textit{is unramified outside} \ S \iff C(\chi) \subset S.
$$	
\end{prop}
\proof  Assume  that $C(\chi) \subset S$ then $\chi(\prod_{v\notin S}\Z_v^*)=\{1\}$. It follows  that if $Z(x)\cap S=\emptyset$ one has $\chi(\prod_{v\in Z(x)}\Z_v^*)=\{1\}$ and \eqref{rami0} shows that $G$ acts freely transitively on $F_x$. Conversely, if $\exists p\in C(\chi) $ with $p\notin S$, then  for $x$ in the periodic orbit $C_p$,   $Z(x)=\{p\}$ fulfills $Z(x)\cap S=\emptyset$ but then, since $\chi(\Z_p^*)\neq \{1\}$, \eqref{rami0} shows that $G$ does not act freely  on $F_x$.\endproof 
The following corollary highlights the \emph{unavoidable} ramified role of the archimedean place.
\begin{cor}\label{ramicor}  There exists a smallest finite set $S$ of places outside which $\pi_\chi$ is unramified. This set is given by $C(\chi)\cup {\infty}$, where $C(\chi)$ is defined in \eqref{cchi}.
\end{cor}
\proof By Proposition \ref{ramiprop}, it is enough to show  that a finite set $S$ of places outside which $\pi_\chi$ is unramified necessarily contains $\infty$. Assume $\infty\notin S$ and let $a\in \A_\Q$ be the adele defined componentwise as 
$$
a_v=1, \forall v\in S\cup C(\chi),  \ \ a_v=0, \ \forall v\notin S\cup C(\chi).
$$
Let $x\in X_\Q$ be the class of $a$ modulo $\Q^\times\times {\widehat\Z}^*$. Let us show that the fiber $\pi_\chi^{-1}(x)$ is reduced to a single element. Let $T=S\cup C(\chi)$ and $(z,g)\in \A_\Q\times G$ be such that the class of $z$ modulo $\Q^\times\times {\widehat\Z}^*$ is $x$. One has $z_v=0$, $\forall v\notin T$ and in particular $z_\infty=0$, moreover there exists $(q,w)\in\Q^\times\times {\widehat\Z}^*$ such that $z_v=qw_v$ for all $v\in T$. Then, for any integer $n$ with $(n,p)=1$ for all $p\in T$ one  has 
$$
z_v=\frac qn (n_vw_v), \ \forall v\in T.
$$
Thus, modulo $\Q^\times\times {\widehat\Z}^*$ one obtains, $$
(z,g)\sim ((a_vn_vw_v),g)\sim (a,\chi((n_vw_v)g).
$$
Since one can choose the integer $n$ such that its residue modulo the conductor of $\chi$ is arbitrary, this shows that all elements of the fiber $\pi_\chi^{-1}(x)$ are equivalent. \endproof 
\begin{rem}\label{quotienttop} The subsets of $\A_\Q$ of the form, for $S\subset \Sigma_\Q$ finite, 
$$
\Omega_S:=\{a\in \A_\Q\mid a_v\neq 0,\ \forall v\in S\}
$$
   are both open (for  the topology  of $\A_\Q$) and invariant under the multiplicative  action of ideles, hence they define open sets in $X_\Q$.  The ordered collection of these  open sets   is similar to the ordered collection of  open sets in the Zariski topology of $\overline{\Spec \Z}$.
\end{rem}
\subsection{Monodromy of $\pi_\chi$ over $C_p$} 
Next, we  transpose Proposition \ref{mappingtorus} to the finite covers $X_\Q^\chi\stackrel{\pi_\chi}{\longrightarrow}X_\Q$ and prove their non-triviality.
\begin{thm}\label{mappingtorus1} Let $X_\Q^\chi\stackrel{\pi_\chi}{\longrightarrow}X_\Q$ be as in Definition \ref{coverdef} and $p\notin C(\chi)$ a non-archimedean prime.\newline
$(i)$~The restriction of $\pi_\chi$ to $\pi_\chi^{-1}(C_p)$ is a $G$-principal bundle over the periodic orbit $C_p$.\newline
$(ii)$~The isomorphism $\tilde \eta$ as in \eqref{eta} induces a canonical isomorphism of $\pi_\chi^{-1}(C_p)$ with the mapping torus of the multiplication by $\chi(p)\in G$, where $p$ is viewed as an element of $\prod_{\ell \in C(\chi)} \Z_\ell^*$.\newline 
$(iii)$~The  monodromy of $C_p$ in $\pi_\chi^{-1}(C_p)$  is given by the action of the element $\chi(p)$.	
\end{thm}
\proof By construction $X_\Q^\chi:=Y_\Q\times_{\widehat\Z^*}G$ and the projection $\pi_\chi$ fulfills $\pi_\chi(z,g)=\pi(z)$ where $\pi:Y_\Q\longrightarrow X_\Q$ is the projection. We thus get 
\begin{equation}\label{prelim}
\pi_\chi^{-1}(C_p)=\pi^{-1}(C_p)\times_{\widehat\Z^*} G 
\end{equation}
By \eqref{eta}, one gets the isomorphism 
$$
\tilde \eta: p^\Z\backslash \left(H\times \R_+^*\right)\longrightarrow \pi^{-1}(C_p)
$$
where $H=\prod_{v\neq p}\Z_v^*$. Using the map $\chi:H\to G$ and \eqref{prelim} one obtains, applying $-\times_{\widehat\Z^*} G$, an isomorphism 
$$
\tilde \eta_\chi:p^\Z\backslash \left(G\times \R_+^*\right)\to \pi_\chi^{-1}(C_p)
$$
which identifies $\pi_\chi^{-1}(C_p)$ with the mapping torus of the multiplication by $\chi(p)\in G$.\endproof Figure \ref{fig1} shows the general aspect of this mapping torus, depending on the subgroup of $G$ generated by $\chi(p)\in G$.  Finally the next corollary shows the non-triviality of the cover $\pi_\chi$.
\begin{cor}\label{monocor}  Let $\pi_\chi$ be as above. There exists infinitely many primes $p$ for which the restriction of the covering $\pi_\chi$ to $\pi_\chi^{-1}(C_p)$ is non-trivial.
\end{cor}
\proof By Dirichlet's density theorem there are infinitely many primes $p$ such that $\chi(p)\neq 1$.\endproof 

\subsection{The class field  functor}
We first review briefly the classical results of class field theory for the field $\Q$.
The Kronecker-Weber theorem states that every finite abelian extension $L$ of $\Q$ in $\C$ is contained in $\Q\left(\zeta_n\right)$ for some $n$, where $\zeta_n$ is a primitive $n$-th root of $1$.  The smallest $n$ such that $L \subseteq \Q\left(\zeta_n\right)$ is called the conductor $m$ of $L/ \Q$. 
One has a surjective homomorphism  $$(\Z / m \Z)^* \cong \operatorname{Gal}(\Q(\zeta_m) / \Q) \stackrel{\theta}{\to } \operatorname{Gal}(L / \Q).$$
The  action of $r\in (\Z / m \Z)^* $ is by  $\theta(r)(\zeta_m)=\zeta_m^r$.
The Artin map associates to each prime $p$ not dividing $m$ an element $\mathrm{Frob}_p\in \operatorname{Gal}(L / \Q)$ defined as follows. Let  $\wp_L$ be a prime ideal of the ring of integers $\mathcal{O}_L$ in $L$ lying above a rational prime  $p$. By definition the decomposition group of $\wp_L$ is the subgroup of the Galois group $\operatorname{Gal}(L / \Q)$ consisting of automorphisms  that leave $\wp_L$ invariant:
$$
D_p=\{\sigma \in \operatorname{Gal}(L / \Q) \mid \sigma(\wp_L)=\wp_L\} .
$$
The decomposition group $D_p$ acts on the residue field $\mathcal{O}_L / \wp_L$, and the kernel of this action is the inertia group. Since  $L$ is an  abelian extension, the decomposition and inertia groups do not depend on the choice of $\wp_L$ above $p$, and since $p$ is unramified in $L$ the inertia group is trivial. This provides a canonical isomorphism 
$$
D_p\cong \operatorname{Gal}(\mathcal{O}_L / \wp_L / \F_p)
$$ 
and thus a canonically defined element $\mathrm{Frob}_p\in D_p\subset \operatorname{Gal}(L / \Q)$ corresponding to the Frobenius generator of $\operatorname{Gal}(\mathcal{O}_L / \wp_L / \F_p)$. The key fact then is 
\begin{fact}\label{artinrec}
	The Frobenius $\mathrm{Frob}_p$ is equal to the image of $p$ under the map $\theta:(\Z / m \Z)^*\to \operatorname{Gal}(L / \Q)$.
\end{fact}
We let $\Q^{\text{ab}}$ be the union of all the cyclotomic fields $\Q\left(\zeta_n\right)\subset \C$. It is the maximal abelian extension $\Q^{\text{ab}}$ of $\Q$. The map $e:\Q/\Z\to \C$ given by $e(s):=\exp(2\pi i s)$ is a group isomorphism of $\Q/\Z$ with the group of roots of unity in $\C$ and we obtain a canonical isomorphism $\theta$ of $\widehat\Z^*=\varprojlim (\Z / m \Z)^* $ with the Galois group $\operatorname{Gal}(\Q^{\text{ab}}/\Q)$.\newline
Given a finite abelian extension $L$ of $\Q$ there exists an isomorphism $\sigma:L\to \Q^{\text {ab }}$ with a subfield of $\Q^{\text {ab }}$. The morphism $\sigma$ is not unique but its range $\sigma(L)$ is unique thus we identify $L$ with this range. We let $W(L)\subset \widehat\Z^*$ be the subgroup 
\begin{equation}\label{wl}
W(L)=\{g\in \widehat\Z^*\mid g(x)=x, \ \forall x\in L\}
\end{equation}
We can now formulate Theorem \ref{mainintro} in a precise manner. 
\begin{thm}\label{main} For $L$ a finite abelian extension of $\Q$, let $\chi(L):\widehat\Z^*\to G:=\operatorname{Gal}(L/\Q)$ be the canonical continuous surjective group morphism. \newline
  $(i)$~The association $L \mapsto\gamma(L):=\left(\pi_\chi(L): X_{\Q}^\chi \to X_{\mathbb{Q}}\right)$ (see Definition \ref{coverdef}) defines a contravariant functor \[\text{Ab}(\Q) \stackrel{\gamma}{\longrightarrow} \text{Cov}_{\text{ab}}(X_\Q)
  \]
  from finite abelian extensions of $\Q$ to finite abelian covers of the space $X_\Q$. \newline
$(ii)$~The finite set $R$ of places of $\Q$ at which the cover $\gamma(L)$ ramifies is the union of the archimedean place with the set of primes at which $L$ ramifies.\newline
$(iii)$~Let $p \notin R$ then the monodromy of $C_p$ in $\gamma(L)$ is the element of $G$ given by the arithmetic Frobenius $\mathrm{Frob}_p$.\newline
$(iv)$~The connected components of the inverse image of $C_p$ in $\gamma(L)$ are covering circles labeled by the places of $L$ over the prime $p$. 	
 \end{thm}
\proof $(i)$~To define the functor $\gamma$ on morphisms, we utilize a specific, uniquely determined decomposition of each morphism as a composition of the canonical inclusion and a Galois automorphism. Let $L_j$, $j=1,2$, be finite abelian extensions of $\Q$, and $\chi_j:\widehat \Z^*\to G_j$ the canonical continuous surjective group morphisms. Let $r:G_2\to G_1$ be the restriction morphism of Galois groups. Next, we plan to  combine the group morphism $r$ with the action of $G_1$ on $\gamma(L_1)$. This leads to consider maps   $\rho:G_2\to G_1$  which fulfill the condition 
\begin{equation}\label{cov}
\chi_1(w)\rho(g)=\rho(\chi_2(w)g), \ \forall w\in \widehat \Z^*, g\in G_2
\end{equation}
To any such map one associates the map $\tilde \rho:X_\Q^{\chi_2}\to X_\Q^{\chi_1}$ defined by
$$
\tilde \rho(z,g)=(z,\rho(g)), \ \forall z\in X_\Q^{ab}, g\in G_2.
$$
We now define the canonical decomposition of morphisms $\sigma:L_1\to L_2$. One can write uniquely $\sigma= \iota\circ k$ where $\iota$ is the canonical  inclusion $\iota :L_1\subset L_2$ and $k\in G_1$. We then define
$$
\rho_\sigma(g):=r(g)k, \ \forall g\in G_2.
$$
One has $r( \chi_2(w))=\chi_1(w)$ for all $w\in \widehat \Z^*$, and this shows that  $\rho_\sigma$ fulfills \eqref{cov}. 

To prove $(i)$, it remains to show that given $\sigma:L_1\to L_2$ and $\sigma':L_2\to L_3$, one has
$$
\rho_{\sigma'\circ \sigma}=\rho_\sigma\circ \rho_{\sigma'}
$$ 
Let $\sigma''=\sigma'\circ \sigma: L_1\to L_3$,  $\iota'':L_1\to L_3$ the canonical inclusion and $r'':G_3\to G_1$ the canonical group morphism (Galois restriction). One has $\iota''=\iota'\circ \iota$, $r''=r\circ r'$. Let now $k\in G_1$ with  $\sigma=\iota\circ k$, and $k'\in G_2$ with  $\sigma'=\iota'\circ k'$. Then one obtains
$$
\sigma''=\sigma'\circ \sigma=\iota'\circ k'\circ \iota\circ k=\iota'\circ \iota\circ r(k')\circ k=\iota''\circ k'', \ k''=r(k') k
$$
where the Galois restriction $r:G_2\to G_1$ fulfills $h\circ \iota=\iota\circ r(h)$ for any $h\in G_2$. By definition one has, for any $g\in G_3$, 
$$
\rho_{\sigma'\circ \sigma}(g)=r''(g)k''=r(r'(g))r(k') k=r(r'(g)k')k=\rho_\sigma(r'(g)k')=\rho_\sigma\circ \rho_{\sigma'}(g)
$$
and this ends the proof of the functoriality.

$(ii)$~The following statement is a `standard'  fact of the relation occurring between local and global class field theory. Let $p$ be a finite prime, then $p$ ramifies in $L$ if and only if  $\chi_L(\Z_p^*)\neq\{1\}$ where $\Z_p^*\subset \widehat \Z^*$ is the local ramification group. Thus  the finite set of $p$ at which $L$ ramifies coincides with $C(\chi)$ as defined in \eqref{cchi}. Then $(ii)$ follows from Corollary \ref{ramicor}.\newline
$(iii)$~Follows from Proposition \ref{mappingtorus1} $(iii)$, and the basic fact \ref{artinrec} showing that the arithmetic Frobenius $\mathrm{Frob}_p\in G$ coincides with $\chi(p)\in G$.\newline
$(iv)$~By $(iii)$ and $(ii)$ of Proposition \ref{mappingtorus1}, the inverse image of $C_p$ is isomorphic to the mapping torus of the multiplication by  $\mathrm{Frob}_p\in G$. Thus $(iv)$ follows from the fact that $p$ being unramified in the extension,   $\mathrm{Frob}_p\in G$ generates the decomposition group $D_p$.\endproof 
\begin{rem}\label{archimedeanplace} One can also investigate the inverse image $\pi_\chi^{-1}(C_\infty)$ where $C_\infty=\{c\}$ and $c$ is the element of $X_\Q$  represented by an adele $a$ such that $a_\infty=0$ and $a_v=1$ for any $v<\infty$. The elements of $\pi_\chi^{-1}(C_\infty)$ are  classes modulo $\widehat \Z^*$ of pairs $(z,g)\in Y_\Q\times G $ such that $\pi(z)=c$. This equality shows that the class of $z$ in $Y_\Q=\Q^\times\backslash \A_\Q$ contains an element $\tilde z\in \A_\Q$ with 
$$
\tilde z_\infty=0, \ \ \tilde z_v\in \Z_v^*, \ \forall v<\infty
$$
 This lift $\tilde z\in \A_\Q$ is uniquely determined up to sign 	since $-1$ is the only rational number $q$ with $\vert q\vert_v=1$ $\forall v<\infty$. Thus the fiber $F_c$ is the quotient $G/\chi(-1)$ of the group $G$ by the element $\chi(-1)$. This element of the Galois group is the complex conjugation, and thus the quotient $G/\chi(-1)$ is identical with the set of archimedean places together with the transitive  Galois action of $G$.
\end{rem}

\section{Schemes and adelic spaces: a generalized class field connection}\label{sec3}
In this section, we establish a conceptual bridge between the adelic spaces arising from our extension of class field theory and the geometric spaces produced by Grothendieck’s extension of Galois theory. This link generalizes the classical correspondence between Galois groups and adelic groups, offering a unified perspective on field extensions.

We denote by $\Z_{(p)}$ the local ring describing the stalk of the structure sheaf of $\Spec\,\Z$ at the point $p$ of the spectrum, and by
$$
r:\Z_{(p)}\to \F_p
$$
the natural residue ring homomorphism.
In this section we  determine the abelianized fundamental group $\pi_1^{et}(\Spec\,\Z_{(p)})^{ab}$,
 identify this group with the symmetry group of the fiber $\pi^{-1}(C_p)\subset Y_\Q$ 
over the periodic orbit $C_p$ and show that the monodromy  of $C_p$ in $\pi^{-1}(C_p)\subset Y_\Q$
is the image of the canonical (Frobenius) generator under the map 
$$r^*:\pi_1^{e t}\left(\operatorname{Spec} \mathbb{F}_p\right)\to\pi_1^{e t}(\operatorname{Spec} \mathbb{Z}_{(p)})^{a b}$$ 
One knows that the \'etale fundamental group $\pi_1^{et}(\Spec\,\F_p)$ is isomorphic to the Galois group ${\rm Gal}(\bar \F_p/\F_p)$ of the algebraic closure $\bar \F_p$ of $\F_p$. The latter is an abelian  profinite group isomorphic to $\widehat\Z$: the isomorphism maps the  Frobenius automorphism acting on $\bar\F_p$ as  $x\mapsto x^p$, to $1\in \widehat\Z$.

Let $ A = \mathbb{Z}_{(p)}$, and let $A^{\mathrm{sh}}$ denote the strict henselization of $A$. This is the integral closure of $A$ in the $p$-typical Witt ring $ W(\overline{\mathbb{F}}_p)$ which is a complete discrete valuation ring. Its field of fractions is the completion of  $\mathbb{Q}_p^{\mathrm{ur}} := \mathbb{Q}_p(\{\zeta_\ell\}_{(\ell, p)=1}) $. 

Let $K$ denote the field of fractions of $A^{\mathrm{sh}}$. The following field inclusions hold:  
\[
\mathbb{Q}(\{\zeta_\ell\}_{(\ell, p)=1}) \subset K \subset \mathbb{Q}_p^{\mathrm{un}}.
\]  
The etale fundamental group $\pi_1^{\text{et}}(\operatorname{Spec}\,\mathbb{Z}_{(p)})$ is known to fit into the exact sequence (see \cite{Stacks}, Lemma 58.11.4): 
\[
\pi_1^{\text{et}}(\operatorname{Spec}\,K) \longrightarrow \pi_1^{\text{et}}(\operatorname{Spec}(\mathbb{Q})) \longrightarrow \pi_1^{\text{et}}(\operatorname{Spec}(\mathbb{Z}_{(p)})) \longrightarrow 1.
\]  
Upon abelianization, this sequence becomes:  
\begin{equation}\label{eqa}
\pi_1^{\text{et}}(\operatorname{Spec}\,K)^{\mathrm{ab}} \stackrel{\alpha}{\longrightarrow} \pi_1^{\text{et}}(\operatorname{Spec}(\mathbb{Q}))^{\mathrm{ab}} \longrightarrow \pi_1^{\text{et}}(\operatorname{Spec}(\mathbb{Z}_{(p)}))^{\mathrm{ab}} \longrightarrow 1.
\end{equation} 
From this, the following result is derived:

\begin{prop}\label{lem3.1}  Let  $p$ be a rational prime. Let $\left\{{\rm F r o b}_{p}\right\}\in \pi_1^{e t}(\Spec(\F_p))$ be the canonical generator.
\begin{enumerate}
\item[(i)] The abelianized \'etale fundamental group $\pi_1^{et}(\Spec\,\Z_{(p)})^{ab}$ is isomorphic to the product $\prod_{v\neq p}\Z_v^*$.

\item[(ii)] The image $r^*(\Fr_p)\in \pi_1^{et}(\Spec\,\Z_{(p)})^{ab}$ is given by the element $p$ diagonally embedded in the product $\prod_{v\neq p}\Z_v^*$.
\end{enumerate}
\end{prop}

\begin{proof}
By the Kronecker-Weber theorem, the maximal abelian extension of $\mathbb{Q}$ is the cyclotomic field $ \mathbb{Q}^{\mathrm{ab}}$. The group of roots of unity $\Q/\Z$ is Pontrjagin  dual to the compact additive group $\widehat \Z$. The Galois group $ \operatorname{Gal}(\mathbb{Q}^{\mathrm{ab}}/ \mathbb{Q}) = \widehat{\mathbb{Z}}^*$ acts on  roots of unity by the dual of the action of $\widehat{\mathbb{Z}}^*$ on $\widehat{\mathbb{Z}}$ by multiplication.

Let $K^{\mathrm{ab}}$ denote the maximal abelian extension of $ K $. By construction one has  
\[
\pi_1^{\text{et}}(\operatorname{Spec}(K))^{\mathrm{ab}} = \operatorname{Gal}(K^{\mathrm{ab}} / K).
\]
Let us simplify the notation by setting $ G = \operatorname{Gal}(K^{ab}/K) $. To determine its image in $ \operatorname{Gal}(\mathbb{Q}^{ab}/\mathbb{Q}) $, it suffices to analyze the action of $G$ on roots of unity. Since $\mathbb{Q}(\{\zeta_\ell\}_{(\ell, p)=1}) \subset K$, $G$ fixes all roots of unity of orders coprime to $p$.
Furthermore, as $K \subset \mathbb{Q}_p^{\text{un}}$ and adjoining a primitive $p^n$-th root of unity to $ \mathbb{Q}_p^{\text{un}}$ produces an abelian Galois extension of degree $(p-1)p^{n-1}$ with Galois group $ (\mathbb{Z}/p^n\mathbb{Z})^* $, the same property holds for $K$. 
This relationship can be visualized using the following diagram of towers of abelian field extensions, where vertical bars represent inclusions.
\[
   \begin{tikzcd}
       K(\{\zeta_{p^\infty}\})\arrow[r,symbol=\subset]\arrow[d,dash,"{\Z_p^*}"'] & \Q_p^{un}(\{\zeta_{p^\infty}\})\arrow[d,dash]\arrow[d, dash,  "{\Z_p^*}"]\\
K\arrow[r,symbol=\subset]\arrow[d,dash] & \Q_p^{un}\arrow[d,dash]\\
       \mathbb{Q}(\{\zeta_\ell\}_{(\ell, p)=1})& \Q_p
   \end{tikzcd}
   \]
In particular, from the inclusion $K(\{\zeta_{p^\infty}\})\subset K^{ab}$, one obtains that  $G$ surjects onto the Galois group $ \mathbb{Z}_p^*$ of the extension of $ K$ obtained by adjoining all roots of unity of $p$-power. Furthermore, by construction $\Q^{ab}\subset K(\{\zeta_{p^\infty}\})$ 
and this ensures that the image of $G$ under the map $\alpha$ in \eqref{eqa} is precisely the $ \mathbb{Z}_p^* $ factor within $ \prod_v \mathbb{Z}_v^* = \operatorname{Gal}(\mathbb{Q}^{ab}/\mathbb{Q}) $. The isomorphism  $\pi_1^{et}(\Spec\,\Z_{(p)})^{ab}\simeq \prod_{v\neq p}\Z_v^*$ then follows from the exactness of  \eqref{eqa}.\newline
(ii)~The action of $\text{Frob}_{p}$ on roots of unity is given by raising to the power $p$.
\end{proof}

With the notations as in section~\ref{sect2.1}, we can now match the (number-theoretic) results related to the abelianized \'etale fundamental group of the  scheme $\Spec\,\Z_{(p)}$ with the geometry inherent to the adelic cover $Y_\Q \stackrel{\pi}{\longrightarrow} X_\Q$.

\begin{thm}\label{main2}
	 Let  $p$ be a rational prime. Let $\left\{{\rm F r o b}_{p}\right\}\in \pi_1^{e t}(\Spec(\F_p))$ be the canonical generator. 
     \begin{enumerate}
         \item[(i)] The inverse image $\pi^{-1}(C_p)\subset Y_\Q$ of the periodic orbit $C_p$  is canonically isomorphic to the mapping torus of the multiplication by $r^*\left\{{\rm F r o b}_{p}\right\}$  in the abelianized \'etale fundamental group $\pi_1^{e t}(\Spec \, \Z_{(p)})^{ab}$.
          \item[(ii)] The canonical isomorphism in $(\rm i)$ is equivariant for the action of the idele class group.
         \item[(iii)] The monodromy of the periodic orbit $C_p$ in $\pi^{-1}(C_p)\subset Y_\Q$ is equal to the natural map $r^*:\pi_1^{e t}\left(\operatorname{Spec} \mathbb{F}_p\right) \rightarrow \pi_1^{e t}(\operatorname{Spec} \mathbb{Z}_{(p)})^{a b}$, and determines the linking of the prime $p$ with all other primes.
         \end{enumerate}
     \end{thm}
     \begin{proof} The proof of this theorem follows from Proposition~\ref{mappingtorus} and the isomorphism given by Proposition~\ref{lem3.1} (i)  $\pi_1^{et}(\Spec\,\Z_{(p)})^{ab}\simeq \prod_{v\neq p}\Z_v^*$ together with the statement (ii).
     \end{proof}

Theorem \ref{main2} elucidates the geometric aspects of the well-known analogy between primes and knots through  the role of the adelic space $X_\Q$, its periodic orbits and their liftings to the adele class space of the rationals $Y_\Q$ viewed as maximal  abelian cover of $X_\Q$.

     The following table summarize what we have discussed so far and in the introduction.

\begin{center}
\renewcommand{\arraystretch}{0.5}
 \begin{tabular}{c|c}
& \\ \textbf{Etale} & \textbf{Adelic} \\
& \\ 
\hline 
& \\$\pi_1^{et}(\operatorname{Spec}\Q)^{ab}={\rm Gal}(\Q^{ab}/\Q)$  & $\left({\rm GL}_1(\A_\Q)/\Q^\times\right)/\left({\rm GL}_1(\A_\Q)/\Q^\times\right)_0$ \\
& \\ 
& \\ Local  Weil group  & $\Q_v^*=$  isotropy  in $Y_\Q=\Q^\times\backslash\A_\Q$  \\
& \\
& \\ $r^*: \operatorname{Spec} \mathbb{F}_p \hookrightarrow \operatorname{Spec} \mathbb{Z}$ & Periodic orbit $C_p$ in $X_\Q$ \\ & \\
& \\ $\pi_1^{e t}\left(\operatorname{Spec} \mathbb{F}_p\right)$ &  $\pi_1(C_p)$~ topological\\ & \\ 
& \\ $\pi_1^{e t}(\operatorname{Spec} \mathbb{Z}_{(p)})^{a b}\simeq \prod_{q\neq p} \mathbb{Z}_q^*$ & $\operatorname{Aut}(\pi^{-1}(C_p))=\prod_{q\neq p} \mathbb{Z}_q^*$  \\
 & \\ 
 & \\ $\pi_1^{e t}\left(\operatorname{Spec} \mathbb{F}_p\right)\to\pi_1^{e t}(\operatorname{Spec} \mathbb{Z}_{(p)})^{a b}$ & Monodromy of $C_p$ in $\pi^{-1}(C_p)\subset Y_\Q$\\
& \\
& \\ $\overline{\operatorname{Spec}\Z}$ ~extended to $\overline{\mathbb S}$  & ${\mathcal X}_\Q=$ first 2 layers of stratification of $X_\Q$\\
& \\
& \\ Generic point $\eta\in \operatorname{Spec}\,\Z$ & Generic orbit $\eta\subset X_\Q$ \\
& \\
\end{tabular}\end{center}
\vspace{0.3in}

\section{The  semilocal description of $X_\Q$ in NCG}\label{sect4}

The space $ X_\Q $, when viewed through the lens of functional analysis, is best understood by focusing on a finite subset of places at a time. Specifically, this involves considering a finite subset $ S \subset \Sigma_\Q $ of places of $ \Q $, which includes the archimedean place $ \infty $. This approach stems from the fact that the function space associated with $ X_\Q $ is expressed as a cross product and comprises functions defined on the adeles and more specifically belonging to the Bruhat–Schwartz space characterized by their non-trivial dependence on only finitely many places. The Bruhat–Schwartz functions $f$ are finite linear combinations of the products $f=\otimes_v f_v$ over  places $v$ of $\Q$, where each $f_v$ is a Bruhat–Schwartz function on the local field $\Q_v$ and $f_v=\mathbf{1}_{\mathcal{O}_v}$ is the characteristic function on the ring of integers $\mathcal{O}_v$ for all but finitely many $v$. Except for a finite set $S$ of places, the functions reduce to the characteristic functions of the maximal compact subring of the restricted infinite product of local fields $ \prod^{\prime}_{v \notin S} \Q_v $. Consequently, the functional analysis focuses on the locally compact rings  
\[
\A_S = \prod_{v \in S} \Q_v.
\]

Moreover, the group involved in the cross product, which globally corresponds to $ \Q^\times $, naturally restricts to a subgroup $ \Gamma := \Z_S^\times \subset \Q^\times $ that depends on the finite set $ S $. After presenting the semilocal setup in section \ref{semilocprep}, we show in section \ref{sheafsemiloc} that the semilocal algebras form a sheaf of algebras over $\Spec\Z$. We then describe in section \ref{sectstrati} the natural stratification of the semilocal spaces  $ Y_{\Q,S} $ and $ X_{\Q,S} $. In section \ref{KC*} we compute the $K$-theory of the $C^*$-algebra associated to the first two stratas.
\subsection{The semilocal setup}\label{semilocprep}
The ring $ \A_S $ includes $ \Q $ as a subring via the diagonal embedding. Let $ \Z_S $ denote the 
ring $\Z$ localized at $S \setminus \{\infty\}$. It is the subring of $ \Q $ consisting of rational numbers 
whose denominators involve only primes $ p \in S $. Explicitly,  
\[
\Z_S = \{ q \in \Q \mid |q|_v \leq 1 \,, \ \forall v \notin S \}.
\]  
The group $ \Gamma_S$, of the invertible elements of $ \Z_S $, can be expressed as  
\[
\Gamma_S =  \{ \pm p_1^{n_1} \cdots p_k^{n_k} \mid p_j \in S \setminus \{\infty\}, \ n_j \in \Z \}.
\]  
This multiplicative group is generated by $ \pm 1 $ and the primes in $ S $. Structurally, it is the direct product of a copy of $ \Z $ for each prime in $ S $ and the group $ \Z/2\Z $ corresponding to the sign $ \pm 1 $. Conceptually, $ \Gamma_S $ is the group $\mathbf{G}_m(S^c)$ of sections of the sheaf $\mathbf{G}_m$  on the open set $S^c\subset \Spec \Z$, complement of $S$.

The semilocal version $ Y_{\Q,S} $ of the adèle class space is defined as the quotient 
\[
Y_{\Q,S} := \Gamma_S \backslash \A_S,
\]  
while the semilocal version of $ X_\Q $ is given by  
\[
X_{\Q,S} := \Gamma_S \backslash \A_S / \widehat{\Z}^*(S), \quad \widehat{\Z}^*(S) = \prod_{v \in S \setminus \{\infty\}} \Z_v^*.
\]

All the components of the global setting, including $ \Q $, $ \Q^\times $, the adeles $ \A_\Q $, the compact ring $ \widehat{\Z} $, its unit group $ \widehat{\Z}^* $, the idele class group $ C_\Q := \mathrm{GL}_1(\A_\Q)/\Q^\times $, and the double quotient $ X_\Q = \Q^\times \backslash \A_\Q / \widehat{\Z}^* $, admit a natural semilocal counterpart. These semilocal analogs are summarized in the following table:

\medskip
\begin{center}
\renewcommand{\arraystretch}{0.5}
 \begin{tabular}{c|c}
& \\ \textbf{Global} & \textbf{Semilocal} \\
& \\ 
\hline
\\
 $ \Q$  &  $\Z_S=\{q\in \Q\,:\, |q|_v\leq 1\,,\ \forall v\notin S\}$ \\
  &\\
  & \\
 $\Q^\times$ &     $\Z_S^\times=\{ \pm p_1^{n_1} \cdots p_k^{n_k} \, :\,  p_j
\in S \setminus\{ \infty \} \,,\, n_j\in \Z\}$\\ & \\
   & \\
$\A_\Q$ &  $\A_{S}=\prod_{v\in S} \Q_v$ \\  & \\
  & \\
     $\widehat \Z$ &     $ \widehat \Z(S)=\prod_{v\in S\setminus\{\infty\}}\Z_v$\\ & \\
    & \\
     $\widehat \Z^*$ &     $ \widehat \Z^*(S)=\prod_{v\in S\setminus\{\infty\}}\Z_v^*$\\ & \\
    & \\
     $C_\Q$ &     $C_{\Q,S}:={\rm GL}_1(\A_S) /\Z_S^\times$\\ & \\
    & \\
     $X_\Q$ &     $X_{\Q,S}:= \Z_S^\times\backslash \A_{S}/ \widehat \Z^*(S)$\\ & \\
     \end{tabular}
\bigskip
\end{center}
\subsection{Sheaf on $\Spec\Z$ of semilocal algebras}\label{sheafsemiloc}
 		 We shall first characterize Bruhat-Schwartz functions of the form $f=1_{
         \Z_v}\otimes g$ by the following properties which are "local" at $v$.
\begin{lem}\label{localv} Let $F$ be a finite set of places $F\subset\Sigma_\Q$, with $\infty\in F$. Let $v\in F$ be non-archimedean. Then a Bruhat-Schwartz function $f\in \cS(\A_F)$ is of the form $f=1_{\Z_v}\otimes g$ for some $g\in \cS(\A_{F\setminus\{v\}})$ if and only if it fulfills the following two properties
\begin{equation}\label{local1}
  a\in \A_F, \ \vert a_v\vert >1~ \Rightarrow ~f(a)=0  
\end{equation}
\begin{equation}\label{local2}
  f(a+ \alpha)=f(a), \ \forall \alpha \in \Q_v, \ \vert \alpha \vert \leq 1.  
\end{equation}    
\end{lem}
\proof Let $f=1_{\Z_v}\otimes g$, then \eqref{local1} holds, 
the second property \eqref{local2} holds since $\Z_v$ is a group under addition and its characteristic function is invariant under the group translations.\newline 
Conversely now, assume that a Bruhat-Schwartz function $f$ fulfills the above two properties. Let us consider the function $g(y):=f(0,y)$ for $y\in \A_{F\setminus\{v\}}$, it is a Bruhat-Schwartz function.  We show that $f=1_{\Z_v}\otimes g$. Let $(x,y)\in \Q_v\times \A_{F\setminus\{v\}}$. If $\vert x\vert >1$, then both $f$ and $1_{
\Z_v}\otimes g$ vanish (using for $f$ the first property \eqref{local1}). If $\vert x\vert \leq 1$, then by the second property \eqref{local2} one has 
$$f(x,y)=f(0,y)=g(y)=(1_{\Z_v}\otimes g)(x,y).$$
 The desired factorization  follows. \endproof 
 		The Zariski topology of  $\Spec\Z$ allows one to better understand the role of the semilocal Bruhat-Schwartz algebras and the transitions maps.
 		\begin{prop}\label{structure} For a finite set of places $S\subset\Sigma_\Q$, with $\infty\in S$, let  $\cS(\A_S)$ denote the Bruhat-Schwartz space of functions on the semilocal adeles $\A_S$. 
 		\begin{enumerate}
 			\item  For any pair of finite sets $\infty\in S \subset S'\subset \Sigma_\Q$, let $R_{S',S}$ be the maximal compact subring of $\prod_{v\in S'\setminus S} \Q_v$. Then the  natural maps 
            $$\gamma(S,S'):\cS(\A_S)\to \cS(\A_{S'})\quad  f\mapsto 1_{R_{S',S}}\otimes f$$  
            define a sheaf $\sheaf$ of algebras on  $\Spec\Z$.
 			\item The stalk of the sheaf $\sheaf$ at the generic point $\eta\in \Spec\Z$ is the  Bruhat-Schwartz algebra of functions on $\A_\Q$.
 			\item The global sections $\Gamma(\Spec\Z,\sheaf)$ form the Schwartz space $\cS(\R)=\cS(\A_{\{\infty\}}$.
 		\end{enumerate}
\end{prop}
\begin{proof} $(1)$~A presheaf of sets on $\Spec\Z$ is a contravariant functor on the small category of  Zariski open subsets of $\Spec\Z$.  This is the category whose objects are  the empty set and  the complements $F^c$ of finite subsets $F\subset \Spec\Z$. Equivalently,  a presheaf of sets on $\Spec\Z$ can be seen as  a covariant functor from the category of finite sets of places $S\ni \infty$. By construction, the maps $\gamma(S,S')$ define morphisms of algebras  which respect the composition of inclusions. Thus $\sheaf$ is a presheaf of algebras. In order to test the sheaf condition  we first consider  the stalk at the generic point, \ie 
$$
\sheaf_\eta:= \varinjlim \cS(\A_S).
$$ 
 By construction this is the  global Bruhat-Schwartz space $\cS(\A_\Q)$ of functions on adeles. Thus, in particular, $(2)$ holds. The fact that the restriction maps $\gamma(S,S')$ are injective allows us to view  $\sheaf$ as a sub-presheaf of the constant sheaf associated to $\cS(\A_\Q)$.  Here we use  the connectedness of $\Spec\Z$ to state that the sections of the constant sheaf on any non-empty open subset coincide with $\cS(\A_\Q)$. To conclude that $\sheaf\subset \cS(\A_\Q)$ is a sheaf one applies the following fact 
\begin{prop*}[\cite{MM}, II, Proposition 1]
If $F$ is a sheaf on a topological space $X$, then a subfunctor $S \subset F$ is a subsheaf if and only if, for every open set $U$ and every element $f \in F(U)$, and every open covering $U=\bigcup U_i$, one has $f \in S(U)$ if and only if $\left.f\right|_{U_i} \in S(U_i)$ for all $i$.
\end{prop*}

Thus, to prove $(1)$ it suffices to show  that for any given  finite sets of places $S_j\ni \infty$, with $S=\cap S_j$,  and for any  $f\in \cS(\A_\Q)$, $f$ belongs to $\sheaf(S^c)$ when $f$ belongs to each $\sheaf(S_j^c)$.

Let $f\in \cS(\A_\Q)$, there exists a finite set $F\supset S$ of places such that $f\in \cS(\A_F)$. 
The maximal compact subring $R_{F\setminus S}$ of $\prod_{v\in F\setminus S} \Q_v$ is the product of the maximal compact subrings $R_v=\Z_v\subset \Q_v$, thus the characteristic function $1_{R_{F\setminus S}}$ is the tensor product 
$\otimes 1_{R_v}$. Therefore one is reduced  to show that for each $v \in F\setminus S$ the Bruhat-Schwartz function $f$ on  $\Q_v\times \A_{F\setminus\{v\}}$ is of the form $1_{R_v}\otimes g$, for some Bruhat-Schwartz function $g$ on $\A_{F\setminus\{v\}}$. Since $S=\cap S_j$ one has $v\notin S_j$ for some $j$ and thus since $f\in\sheaf(S_j^c)$, Lemma \ref{localv} gives the desired factorization. \newline
$(3)$~The global sections ${\Gamma(\Spec\Z,\sheaf)}$ correspond to the set $S=\{\infty\}$. 
\end{proof} 

The identification 
$ \Gamma_S =\mathbf{G}_m(S^c)$ 
supplies the conceptual meaning of the groups $ \Gamma_S $ as sections of the sheaf in groups $\mathbf{G}_m$  on  $\Spec\Z$. 

One can then easily transpose 
Proposition \ref{structure} to the algebraic cross products as follows 
\begin{thm}\label{structure1} For a finite set of places $\Sigma_\Q\supset S\ni \infty$, let $\cS(\A_S)\rtimes \Z_S^\times$ be the algebraic cross product of the Bruhat-Schwartz algebra $\cS(\A_S)$ by the action of $  \Z_S^\times$. 
 		\begin{enumerate}
 			\item The algebraic cross product $\sheaf\rtimes \mathbf{G}_m$ defines a sheaf  of algebras on  $\Spec\Z$ such that 
 			$$
 			\left(\sheaf\rtimes \mathbf{G}_m\right)(S^c)=\cS(\A_S)\rtimes \Z_S^\times.
 			$$
 			\item The stalk of the sheaf $\sheaf\rtimes \mathbf{G}_m$ at the generic point $\eta\in \Spec\Z$ is the global cross product $\cS(\A_\Q)\rtimes \Q^\times$. 			\item The global sections of $\sheaf\rtimes \mathbf{G}_m$ form the cross product $\cS(\R)\rtimes \{\pm 1\}$.
 		\end{enumerate}
\end{thm}
\proof By construction the algebraic cross products $\cS(\A_S)\rtimes \Z_S^\times$ define a presheaf $\sheaf\rtimes \mathbf{G}_m$ of algebras which is a sub-presheaf of the constant sheaf given by the global cross product $\cS(\A_\Q)\rtimes \Q^\times$. To prove that $\sheaf\rtimes \mathbf{G}_m$ is a sheaf we proceed as in the proof of Proposition \ref{structure}. Let $h\in \cS(\A_\Q)\rtimes \Q^\times$, then $h$ can be uniquely written as a finite sum 
$$
h=\sum f_j U_{q_j}, \ \ q_j\in \Q^\times, \ f_j \in \cS(\A_\Q), \ f_j\neq 0.
$$
If $h\in \left(\sheaf\rtimes \mathbf{G}_m\right)(S_k^c)$ for finite sets of places $S_k\ni \infty$ with $\cap S_k=S$, one has 
$q_j\in \cap \,\mathbf{G}_m(S_k^c)$ and hence $q_j\in \mathbf{G}_m(S^c)$ for all $j$. Moreover, for all $j$ one has $f_j\in \sheaf(S_k^c)$ for all $k$ and Proposition \ref{structure} shows that $f_j\in \sheaf(S^c)$. This shows that $\sheaf\rtimes \mathbf{G}_m$ is a sheaf. The proof of the other statements is straightforward. \endproof

\subsection{The stratification of the space $ Y_{\Q,S} $}\label{sectstrati}
 For $ x \in X_{\Q,S} $, represented by the adele $ a = (a_v) \in \A_S $, the set of places  
\[
Z(x) = \{v \in S \mid a_v = 0\}  
\]  
is well-defined and does not depend on the choice of the representative $ a $, as noted after \eqref{zxdef}. The same property holds for $ x \in Y_{\Q,S} $. For each $ v \in S $, the subset defined by the condition $ v \in Z(x) $ is closed in the quotient topology.  

Define $ \nu(x) := \# Z(x) $. For any integer $ n $, the condition $ \nu(x) \geq n $ defines a closed subset of both $ Y_{\Q,S} $ and $ X_{\Q,S} $ and open subsets  
\[
Y_{\Q,S}^{(n)} := \{x \in Y_{\Q,S} \mid \# Z(x) < n\}, \quad  
X_{\Q,S}^{(n)} := \{x \in X_{\Q,S} \mid \# Z(x) < n\}.  
\]  
We now proceed to describe the natural stratification of the space $ Y_{\Q,S} $.

\begin{prop}\label{ergodic23} Let  $\pi_S:Y_{\Q,S}\to X_{\Q,S}$ denote the quotient map. Then 
\begin{enumerate}
\item[(i)] The orbits of the action of the group $C_{\Q,S}$ on $Y_{\Q,S}$ are indexed by the subsets of $S$ as follows: 
$$
\Omega_Z:=\Gamma\backslash\{(a_v)\in \A_S\mid a_v= 0\qqq v \in Z, \ a_v\neq 0\qquad \forall v \notin Z\}\qqq Z\subset S.
$$
\item[(ii)] The orbit  $\Omega_{\{p\}}$ associated to the subset $\{p\}\subset S$ is  the inverse image $\pi_S^{-1}(C_p)$ of the periodic orbit  $C_p$ of length $\log p$.
\item[(iii)] The orbit $\pi_S^{-1}(C_p)$ is equivariantly isomorphic to the mapping torus of the multiplication by $p$ in the compact group $ \prod_{q\in S,q\neq p} \Z_q^*$.	
\end{enumerate}
\end{prop}
\begin{proof} 
 (i) This follows from the transitivity of the action of $ \GL_1(\Q_v) $ on $ \Q_v \setminus \{0\} $ for any place $ v $.  

(ii) The periodic orbit $ C_p $, with length $ \log p $ under the action of the scaling group $ \R_+^* $ on $ X_{\Q,S} $, is given by $ \pi_S(\Omega_{\{p\}}) \subset X_{\Q,S} $ and is identical to the quotient $ p^\Z \backslash \R_+^* $.  

(iii) The set $ \Omega_{\{p\}} $ satisfies $ \Omega_{\{p\}} = \pi_S^{-1}(C_p) $. Furthermore, the group $ \Gamma $ decomposes as $ \Gamma = p^\Z \times \Gamma' $, where  
\[
\Gamma' = \{ \pm p_1^{n_1} \cdots p_k^{n_k} \mid p_j \in S \setminus \{p, \infty\}, \ n_j \in \Z \}.
\]  
The action of $ \Gamma' $ on $ \{0\}_p \times \left(\prod_{q \in S, q \neq p} \Q_q \setminus \{0\}\right) \times \R^* $ has a fundamental domain $ \{0\}_p \times \prod_{q \in S, q \neq p} \Z_q^* \times \R_+^* $, which remains invariant under the action of $ p^\Z $.  

The quotient  
\[
p^\Z \backslash \left(\prod_{q \in S, q \neq p} \Z_q^* \times \R_+^*\right)
\]  
is the mapping torus corresponding to multiplication by $ p $ on the compact group $ \prod_{q \in S, q \neq p} \Z_q^* $.  \end{proof}

 This description reveals that $ C_p $ lies at the boundary of the free orbit $ \pi_S(\Omega_{\emptyset}) \subset X_{\Q,S} $. The condition $ \nu(x) < 2 $ defines an open subset $ \Omega $ of $ Y_{\Q,S} $, whose structure, along with its counterpart in $ X_{\Q,S} $, is readily accessible. The latter is the union of the generic orbit with the periodic orbits $ C_v $ for $ v \in S $.

  \subsection{$K$-theory of the $ C^* $-algebra of the two first stratas}\label{KC*}
  We  provide a detailed proof of the following statement, which was stated without proof in \cite{CC8}:  \vspace{.02in}
"The simplest meaningful computation of the $ K $-theory of the $ C^* $-algebras involved pertains to the cross product $ A $ associated with the union in $ X_{\Q,S} $ of the generic orbit and the three periodic orbits $ C_p, C_q, C_\infty $. The result is $ K_0(A) \simeq \Z^3 $, reflecting the presence of the three periodic orbits, while $ K_1(A) \simeq \Z^2 $, capturing the one-dimensionality of the periodic orbits $ C_p $ and $ C_q $."\vspace{.02in} 

We consider the $ C^* $-algebra associated with the open subset formed by the union of the generic orbit and the periodic orbits. As a topological space, this subset is locally compact. We describe it explicitly in the semilocal setting of two primes.  

Let $ S = \{p, q, \infty\} $, and consider the open subspace  
\[
\Omega \subset \A_S = \Q_p \times \Q_q \times \R
\]  
defined by adeles having at most one zero component. Its complement in $ \A_S $ consists of the union of three closed subsets corresponding to adeles with at least two zero components.  

Next, we divide $ \Omega $ by the action of the compact group $ G = \Z_p^* \times \Z_q^* $. This yields a locally compact space  
\[
Z := \Omega / G,
\]  
which is the union of the following four subspaces:

\begin{enumerate}
	\item $Z_\emptyset=\left(\Q_p^*\times \Q_q^*\times \R^*\right)/G\simeq p^\Z\times q^\Z\times \{\pm 1\}\times \R_+^*$.
	\item $Z_p=\left(\{0\}\times \Q_q^*\times \R^*\right)/G\simeq \{0\}\times q^\Z\times \{\pm 1\}\times \R_+^*$.
	\item $Z_q=\left(\Q_p^*\times \{0\}\times \R^*\right)/G\simeq p^\Z\times \{0\}\times \{\pm 1\}\times \R_+^*$.
	\item $Z_\infty=\left(\Q_p^*\times \Q_q^*\times \{0\}\right)/G\simeq p^\Z\times q^\Z\times \{0\}$.
\end{enumerate} 

Next, we describe the cross product $C^*$-algebra $A=C_0(Z)\ltimes \Gamma$.
\begin{lem} The cross product $C^*$-algebras are, up to Morita equivalence, with $\mathcal K$ the compact operators,
\begin{enumerate}
	\item $A_\emptyset=C_0(Z_\emptyset)\ltimes \Gamma=\mathcal K\otimes C_0(\R_+^*)$.
	\item $A_p=C_0(Z_p)\ltimes \Gamma=\mathcal K\otimes C(C_p)$	\item $A_q=C_0(Z_q)\ltimes \Gamma=\mathcal K\otimes C(C_q)$
	\item $A_\infty =C_0(Z_\infty)\ltimes \Gamma= \mathcal K\otimes C^*(\Z/2\Z)$.
\end{enumerate} 	
\end{lem}

One has an exact sequence of $C^*$-algebras of the form 
$$
0\to A_\emptyset\stackrel{\iota}{\to} A\stackrel{\rho}{\to} \left(A_p\oplus A_q\oplus A_\infty\right)\to 0 
$$
None of the involved $C^*$-algebras is unital but the $C^*$-algebra $B=A_p\oplus A_q\oplus A_\infty$ on the right in the above sequence is Morita equivalent to the unital $C^*$-algebra $\mathbf C:=C(C_p)\oplus C(C_q)\oplus C^*(\Z/2\Z)$.
We thus get the exact hexagon of $K$-theory groups
\begin{equation}\label{sixterm}
\begin{gathered}
\xymatrix@C-1em{
& K_0(A_\emptyset) \ar[rr]^{\iota_*} && K_0(A) \ar[dr]^{\rho_*} \\
K_1(B) \ar[ur]^{\delta_1} &&&& K_0(B) \ar[dl]^{\delta_0} \\
& K_1(A) \ar[ul]^{\rho_*} && K_1(A_\emptyset) \ar[ll]^{\iota_*}
}
\end{gathered}
\end{equation}
One has
\begin{enumerate}
	\item $K_0(A_\emptyset)=K_0( C_0(\R_+^*))=\{0\}$. $K_1(A_\emptyset)=K_1( C_0(\R_+^*))=\Z$.
	\item $K_0(A_p)=K_0(C(C_p))=\Z$, $K_1(A_p)=K_1(C(C_p))=\Z$.
	\item $K_0(A_q)=K_0(C(C_q))=\Z$, $K_1(A_q)=K_1(C(C_q))=\Z$.
	\item $K_0(A_\infty) =K_0(C^*(\Z/2\Z))=\Z^2$, $K_1(A_\infty) =K_1(C^*(\Z/2\Z))=\{0\}$.
\end{enumerate}

 From this, we deduce that $ K_0(B) = \Z^4 $ and $ K_1(B) = \Z^2 $. Starting with $ K_0(A_\emptyset) = \{0\} $, it follows that the map  
\[
\rho_*: K_0(A) \to K_0(B)
\]  
is injective. To determine the range of this map, which is a subgroup $ R \subset \Z^4 $ and hence a free abelian group, we need to analyze the connecting map  
\[
\delta_0: K_0(B) \to K_1(A_\emptyset) = \Z.
\]  

\begin{lem}\label{connecting} Let $ C := C_0(\R) \ltimes \Z/2\Z $ denote the cross product by the symmetry $ x \mapsto -x $, and let $ r: C \to C^*(\Z/2\Z) $ be the morphism associated with evaluation at $ 0 \in \R $. Then, with $ J := \ker(r) $, the connecting map  
\[
\delta_0: K_0(C^*(\Z/2\Z)) \to K_1(J) = \Z  
\]  
is surjective. 
\end{lem}

 \begin{proof}
    An element of $ C = C_0(\mathbb{R}) \ltimes \mathbb{Z}/2\mathbb{Z} $ can be written as $ f + gU $, where $ f, g \in C_0(\mathbb{R}) $, $ U^2 = 1 $, and $ UfU^*(x) = f(-x) $ for all $ x \in \mathbb{R} $. To this element, we associate the matrix-valued function:  
\[
x \mapsto 
\begin{pmatrix}
f(x) & g(x) \\
g(-x) & f(-x)
\end{pmatrix}.
\]
The multiplication rule for these matrices is given by:  
\begin{multline*}
\begin{pmatrix}
f(x) & g(x) \\
g(-x) & f(-x)
\end{pmatrix}\begin{pmatrix}
h(x) & k(x) \\
k(-x) & h(-x)
\end{pmatrix}
= \\
=\begin{pmatrix}
f(x)h(x) + g(x)k(-x) & f(x)k(x) + g(x)h(-x) \\
f(-x)k(-x) + g(-x)h(x) & f(-x)h(-x) + g(-x)k(x)
\end{pmatrix}.
\end{multline*}

This multiplication corresponds to the relation:  
\[
(f + gU)(h + kU) = \big(fh + gk^s\big) + \big(fk + gh^s\big)U,  
\]
where $ k^s(x) = k(-x) $ and $ h^s(x) = h(-x) $. Thus, this construction provides an isomorphism of $ C $ with the sub-$ C^* $-algebra of $ C_0(\mathbb{R}_+, M_2(\mathbb{C})) $ consisting of matrices whose value at $ x = 0 $ lies in $ C^*(\mathbb{Z}/2\mathbb{Z}) \subset M_2(\mathbb{C}) $, i.e., matrices of the form:  
\[
a + bU = 
\begin{pmatrix}
a & b \\
b & a
\end{pmatrix}.
\]
The following equivalence shows the diagonalization of these matrices:  
\[
\begin{pmatrix}
1 & 1 \\
-1 & 1
\end{pmatrix}
\begin{pmatrix}
a & b \\
b & a
\end{pmatrix}
\begin{pmatrix}
\frac{1}{2} & -\frac{1}{2} \\
\frac{1}{2} & \frac{1}{2}
\end{pmatrix}
=
\begin{pmatrix}
a + b & 0 \\
0 & a - b
\end{pmatrix}.
\]
Using a unitary conjugation, we can reduce to the sub-$ C^* $-algebra of $ C_0(\mathbb{R}_+, M_2(\mathbb{C})) $ formed of matrices whose value at $ x = 0 $ is diagonal, i.e., belonging to $ D \subset M_2(\mathbb{C}) $.  

To complete the analysis, we examine the connecting map  
\[
\delta_0: K_0(D) \to K_1(C_0((0, \infty), M_2(\mathbb{C}))).
\]
As shown in \cite{Bl} (page 77), this map is defined by associating to an idempotent $ e \in D $ the loop:  
\[
(0, \infty) \ni x \mapsto \exp(2\pi i f(x)) \in M_2(\mathbb{C}),
\]
where $ f \in C_0(\mathbb{R}_+, M_2(\mathbb{C})) $ satisfies $ f(0) = e $.  

Choosing $ e $ as the element of $ D $ with diagonal values $ (1, 0) $, and extending it as a diagonal matrix with elements $ (t(x), 0) $, where $ t $ is a continuous function varying from $ t(0) = 1 $ to $ t(\infty) = 0 $, produces a generator of $ K_1(C_0((0, \infty), M_2(\mathbb{C}))) $.  

This computation mirrors the case of $ K_1(C_0([0, \infty))) = \{0\} $. Indeed,  the associated six-term exact sequence displays as:  
\[
\begin{gathered}
\xymatrix@C-1em{
& K_0(C_0((0,\infty))) = \{0\} \ar[rr] && K_0(C_0([0,\infty))) \ar[dr] \\
K_1(\mathbb{C}) = \{0\} \ar[ur]^{\delta_1} &&&& K_0(\mathbb{C}) = \mathbb{Z} \ar[dl]^{\delta_0} \\
& K_1(C_0([0,\infty))) = \{0\} \ar[ul] && K_1(C_0((0,\infty))) = \mathbb{Z} \ar[ll]
}
\end{gathered}
\]
thus it is evident that $ \delta_0 $ is surjective.
\end{proof}

 We can now complete the terms in the exact sequence \eqref{sixterm} as follows:
\[
\begin{gathered}
\xymatrix@C-1em{
& K_0(A_\emptyset) = \{0\} \ar[rr]^{\iota_*} && K_0(A) \ar[dr]^{\rho_*} \\
K_1(B) = \mathbb{Z}^2 \ar[ur]^{\delta_1} &&&& K_0(B) = \mathbb{Z}^4 \ar[dl]^{\delta_0} \\
& K_1(A) \ar[ul]^{\rho_*} && K_1(A_\emptyset) = \mathbb{Z} \ar[ll]^{\iota_*}
}
\end{gathered}
\]
Next, we will show that the map $ \delta_0 : K_0(B) \to K_1(A_\emptyset) $ is surjective. Let $ A_1 \subset A $ be the subalgebra (in fact, ideal) of $ A $, which is the cross product $ A_1 = C_0(Z_1) \ltimes \Gamma $, where $ Z_1 = Z_\emptyset \cup Z_\infty $. We have the identification:  
\[
A_1 = \mathcal{K} \otimes \left(C_0(\mathbb{R}) \ltimes \mathbb{Z}/2\mathbb{Z}\right),
\]
and the following commutative diagram:
\[
\xymatrix{
    A_\emptyset \ar[r]^\iota \ar[d]_{\rm Id}  & A_1 \ar[d]^\subset \ar[r]^{\rho_1} & A_\infty \ar[d]^\subset \\
    A_\emptyset \ar[r]_\iota & A \ar[r]_\rho & B
  }
\]
where the map $ \rho_1 : A_1 \to A_\infty $ is given by $ \rm{Id}_{\mathcal{K}} \otimes r $, with $ r $ being the restriction map $ r: C \to C^*(\mathbb{Z}/2\mathbb{Z}) $ as in Lemma \ref{connecting}. By the lemma, we obtain a projection $ e \in A_\infty $, whose image under the connecting map $ \delta_0 $ to $ K_1(A_\emptyset) = \mathbb{Z} $ is the generator of $ \mathbb{Z} $. Specifically, $ \delta_0([e]) = [\exp(2\pi i h)] \in K_1(A_\emptyset) $, where $ \rho_1(h) = e $. We can lift $ e $ to $ A $ by choosing the same $ h \in A $, so we conclude that the connecting map $ \delta_0 : K_0(B) \to K_1(A_\emptyset) $ is surjective.

 \vspace{.1in}

\section{The classifying space $\Gamma\backslash\left(\A_{S}\times\underline{E \Gamma}\right)$ and its codimension $1$ foliation} \label{sect5}

 Let $ S \ni \infty $ be a finite set of places of $ \mathbb{Q} $. The $ C^* $-algebra $ \mathcal{A} = C_0(\mathbb{A}_S) \rtimes \Gamma $, associated with the noncommutative space $ Y_{\mathbb{Q}, S} := \Gamma \backslash \mathbb{A}_S $, is not of type I when $ S $ contains at least two distinct primes (see \cite{CM}, Lemma 2.28, Chapter II). To understand the $ K $-theory of $ \mathcal{A} $, we consider the assembly map \cite{BCH}, which computes the $ K $-theory of the $ C^* $-algebra $ \mathcal{A} $ in terms of the universal proper action $ \underline{E\Gamma} $ of $ \Gamma $. Since $ \Gamma $ is commutative, all definitions of the cross product $ C^* $-algebra coincide with the reduced cross product, so we omit the subscript $ r $ in the notation $ C^*_r $. Moreover, the assembly map is an isomorphism. The group $ \Gamma $ is $ \mathbb{Z}^n \times \{\pm 1\} $, where $ n $ is the number of distinct primes in $ S $, and the universal space for proper actions, $ \underline{E \Gamma} $, is $ \mathbb{R}^n $, where $ \mathbb{Z}^n $ acts by translations and $ \{\pm 1\} $ acts trivially. This leads to the space 
\[
\mathcal{Y}_S = \Gamma \backslash \left( \mathbb{A}_S \times \underline{E \Gamma} \right),
\]
which is the quotient of $ \mathbb{A}_S \times \mathbb{R}^n $ by the diagonal action of $ \Gamma $, given by
\[
\mathcal{Y}_S = \Gamma \backslash \left( \prod_{v \in S} \mathbb{Q}_v \times \mathbb{R}^n \right).
\]
This space resolves the noncommutative nature of the quotient $ \Gamma \backslash \mathbb{A}_S $.

\begin{prop}
\label{bciso}
\begin{enumerate}
\item[(i)] The space $ \mathcal{Y}_S $ is locally compact and has dimension $ \#S $.
\item[(ii)] The action of $ \mathbb{R}^n $ by translations on the space $ \mathcal{Y}_S $ defines a codimension 1 lamination of $ \mathcal{Y}_S $, with the space of leaves being $ Y_S $.
\end{enumerate}
\end{prop}

\begin{proof}
(i) The properness of the action of $ \Gamma $ on $ \underline{E \Gamma} $ implies that the action of $ \Gamma $ on $ \mathbb{A}_S \times \underline{E \Gamma} $ is also proper. Thus, $ \mathcal{Y}_S $ is Hausdorff, and since $ \mathbb{A}_S \times \underline{E \Gamma} $ is locally compact, it follows that $ \mathcal{Y}_S $ is locally compact. The dimension of $ \prod_{v \in S} \mathbb{Q}_v \times \mathbb{R}^n $ is $ n + 1 = \#S $, and this dimension is preserved under the quotient.

(ii) The orbits of the $ \mathbb{R}^n $-action by translation on $ \prod_{v \in S} \mathbb{Q}_v \times \mathbb{R}^n $ define a lamination that is invariant under the action of $ \Gamma $, and hence descends to the quotient $ \mathcal{Y}_S $. The space of leaves of this lamination is $ \prod_{v \in S} \mathbb{Q}_v $, and its quotient by $ \Gamma $ is $ Y_S $.
\end{proof}



\addcontentsline{toc}{section}{References}

\bibliographystyle{amsalpha}

\begin{thebibliography}{A}


  \bibitem[BCH]{BCH} P.~Baum, A.~Connes, N.~Higson,  \emph{ Classifying space for proper actions and $K$-theory of group
$C^*$-algebras}. (English summary)
$C^*$-algebras: 1943-1993 (San Antonio, TX, 1993), 240-291, Contemp. Math., 167, Amer. Math. Soc., Providence, RI, 1994.

  \bibitem[Bl]{Bl} B.~Blackadar, {\em K-theory for operator algebras}, Mathematical Sciences Research Institute Publications (MSRI, volume 5).

   \bibitem[Br]{Br} F.~Bruhat, {\em Distributions sur un groupe localement compact et applications à l'étude des représentations des groupes $p$-adiques}, Bull. Soc. Math. France 89 (1961), 43-75.
 
\bibitem[C]{C} A.~Connes, \emph{Trace formula in noncommutative geometry and the zeros of the Riemann zeta function}.  Selecta Math. (N.S.)  \textbf{5}  (1999),  no. 1, 29--106.


\bibitem[CC1]{CC1} A. Connes, C. Consani, \emph{Schemes over $\mathbb F_1$ and Zeta Functions}, Compos. Math. \textbf{146} (2010), no. 6, 1383--1415.
 
 \bibitem[CC2]{CC2} A.~Connes, C.~Consani  {\em  From monoids to hyperstructures: in search of an absolute arithmetic}. Casimir force, Casimir operators and the Riemann hypothesis, 147--198, Walter de Gruyter, Berlin, 2010.



\bibitem[CC3]{CC3} A.~Connes, C.~Consani, {\em Geometry of the Arithmetic Site}. Adv. Math. 291 (2016), 274--329.

\bibitem[CC4]{CC4} A.~Connes, C.~Consani, \emph{Geometry of the Scaling Site}.  Selecta Math. (N.S.) \textbf{23} (2017), no. 3, 1803--1850.

\bibitem[CC5]{CC5} A.~Connes, C.~Consani, \emph{The Riemann-Roch strategy, complex lift of the Scaling Site},  ``Advances in Noncommutative Geometry, On the Occasion of Alain Connes' 70th Birthday", Chamseddine, A., Consani, C., Higson, N., Khalkhali, M., Moscovici, H., Yu, G. (Eds.), Springer (2020). 

\bibitem[CC6]{CC6} A.~Connes, C.~Consani, \emph{Riemann-Roch for the ring $\Z$}. C. R. Math. Acad. Sci. Paris 362 (2024), 229–235.

\bibitem[CC7]{CC7} A.~Connes, C.~Consani, \emph{Riemann-Roch for $\overline{\text{Spec}\,\Z}$}, Bulletin des Sciences Mathematiques 187 (2023), Paper No. 103293, 29pp.

\bibitem[CC8]{CC8} A.~Connes and C.~Consani, {\em Knots, Primes and the Adele Class Space}, \url{https://arxiv.org/pdf/2401.08401}



\bibitem[CM]{CM} A.~Connes, M.~Marcolli,  \emph{  Noncommutative Geometry, Quantum Fields, and Motives}, Colloquium Publications, Vol.55, American Mathematical Society, 2008.

  \bibitem[HK]{HK} N.~Higson, G.~Kasparov, \emph{ Operator K-theory for groups which act properly and isometrically on Hilbert space}, Electron. Res. Announc. Amer. Math. Soc. 3 (1997), 131-142

\bibitem[L]{L} H.~Lenstra, {\em Galois theory for schemes}. First edition: 1985
(Mathematisch Instituut, Universiteit van Amsterdam)
Second edition: 1997
(Department of Mathematics, University of California at Berkeley)
Electronic third edition: 2008






\bibitem[M]{M} Y.~Manin {\em Introduction into theory of schemes}. Translated from the Russian and edited by Dimitry Leites. Abdus Salam School of Mathematical Sciences Lahore, Pakistan. (2009).

\bibitem[Ma0]{Ma0}  B.~Mazur, {\em Remarks on the Alexander Polynomial}, (1963), \url{https://bpb-us-e1.wpmucdn.com/sites.harvard.edu/dist/a/189/files/2023/01/Remarks-on-the-Alexander-Polynomial.pdf}

\bibitem[Ma1]{Ma1}  B.~Mazur, {\em Primes, Knots and Po}, Lecture notes for the conference "Geometry, Topology and Group Theory" in honor of the 80th birthday of Valentin Poenaru, July 2012. \url{https://www-fourier.ujf-grenoble.fr/~funar/CONFERENCES/article_barry.pdf}

\bibitem[Ma2]{Ma2}  B.~Mazur, {\em Bridges between Geometry and Number Theory},  talk at the conference ``Unifying themes in geometry'' como 2021 \url{https://www.youtube.com/watch?v=U3EzqIYgqEw}

\bibitem[Ma3]{Ma3}  B.~Mazur, {\em Thoughts about primes and knots},  (2021),  \url{https://www.youtube.com/watch?v=KTVEFwRbuzU}

\bibitem[MM]{MM} S.~MacLane, I.~Moerdijk, {\emph Sheaves in geometry and logics}, A first introduction to topos theory. Corrected reprint of the 1992 edition. Universitext. Springer-Verlag, New York, 1994. xii+629 pp.


\bibitem[RM]{RM} R.~Meyer,  {\emph The cyclic homology and K-theory of certain adelic crossed products},  arXiv:math/0312305.

\bibitem[Mo1]{Mo1} M.~Morishita, {\em Analogies between knots and primes, 3-manifolds and number rings} [translation of MR2208305]. Sugaku expositions. Sugaku Expositions 23 (2010), no. 1, 1--30.

\bibitem[Mo2]{Mo2} M.~Morishita, {\em Knots and primes. An introduction to arithmetic topology}. Universitext. Springer, London, 2012.

\bibitem[Stacks]{Stacks}   \emph{The Stacks project} 2024,  \url{https://stacks.math.columbia.edu} 

\end{thebibliography}

\vspace{20pt}
\small
\noindent

\begin{tabular}{ l l l l l l l l l l l l l}
 Alain Connes &&&&&&  Caterina Consani  &&&&&&  ~\\ 
 Coll\`ege de France &&&&&&  Department of Mathematics  &&&&&&  ~\\  3 Rue d'Ulm &&&&&&  Johns Hopkins University  &&&&&& ~\\
 75005 Paris &&&&&&  Baltimore, MD 21218 &&&&&&  ~\\
 France &&&&&&  USA &&&&&&  ~\\
 \href{mailto:alain@connes.org}{alain@connes.org} &&&&&&  \href{mailto:cconsan1@jhu.edu}{cconsan1@jhu.edu}&&&&&&  ~
\end{tabular}

\end{document}